\documentclass[10pt]{elsarticle}
\journal{Elsevier}
\usepackage{amssymb,amsmath,amsthm,bm}
\usepackage{amsfonts}
\usepackage{epsfig}
\usepackage{amsbsy}
\usepackage{amsmath}
\usepackage{color}
\biboptions{sort&compress}
\DeclareMathOperator{\Span}{span}

\newcommand\R{{\mathbb{R}}}

\def\ri{\mathrm{i}}

\usepackage{graphicx}
\graphicspath{{figures/}}

\newcommand{\eqnn}[1]{\begin{equation}\begin{split}#1\end{split}\nonumber\end{equation}}
\newcommand{\beq}{\begin{equation}}
\newcommand{\eeq}{\end{equation}}
\newcommand{\bea}{\begin{eqnarray}}
\newcommand{\eea}{\end{eqnarray}}

\newcommand{\eq}[2]{\begin{equation}\begin{split}#1\end{split}\label{#2}\end{equation}}

\newlength\imagewidth
\newtheorem{proposition}{Proposition}[section]
 \newtheorem{lemma}{Lemma}[section]
 \newtheorem{theorem}{Theorem}[section]

\newtheorem{remark}{Remark}[section]

\newcommand\LT{L^2(\mathbb{R})}

\newcommand\HT{H^1(\mathbb{R})}

\usepackage{color}
\definecolor{dgreen}{rgb}{0,.6,0}

\numberwithin{equation}{section}
\usepackage[normalem]{ulem}
\begin{document}

\begin{frontmatter}
\title{Stability of 2-soliton solutions for the modified
Camassa-Holm equation with cubic nonlinearity }
\author{Xijun Deng}
\ead{xijundeng@yeah.net}
\address{Department of Mathematics,\\ Hubei University of Automotive Technology,
Shiyan, Hubei, 442002, P. R. China}
\author{St\'{e}phane Lafortune}
\ead{lafortunes@cofc.edu}
\address{Department of Mathematics,\\
College of Charleston,
 Charleston, SC 29401, USA }
\author{Zhisu Liu}
\ead{liuzhisu@cug.edu.cn}
\address{Center for Mathematical Sciences, School of Mathematics and Physics, \\ China University of Geosciences, Wuhan,
Hubei, 430074, P. R. China}

\begin{abstract}
 In this paper, we are concerned with the stability of {2-soliton solutions on a nonzero constant background} for the modified
Camassa-Holm equation with cubic nonlinearity. By employing conserved quantities in terms of the momentum variable $m$, we show that the 2-soliton, when regarded as a solution to the
initial-value problem for the modified Camassa-Holm equation, is nonlinearly stable
to perturbations with respect to the momentum variable in the Sobolev space $H^2$.
\end{abstract}


\end{frontmatter}

\section{Introduction}
In this paper, we deal with the nonlinear stability of 2-soliton solutions for the modified Camassa-Holm (mCH) equation with cubic nonlinearity
\begin{equation} \label{1} m_t+\left((u^2-u_x^2)m\right)_x=0,\;\;m=u-u_{xx}, \end{equation}
where $u(x,t)$ is a real-valued function of space-time variables $(x, t)$, and the subscripts $x$ and $t$ appended to $m$ and $u$ denote partial differentiation. The mCH equation \eqref{1} was proposed as an integrable system by Fuchssteiner \cite[Equation (3.6)]{FU96} and Olver and Rosenau \cite[Equation (25)]{OR96} (see also \cite{FO95,FO952}) using an approach based on the idea of ``tri-Hamiltonian duality.'' This approach consists of a recombination of the Hamiltonian operators of a bi-Hamiltonian system that leads to integrable hierarchies endowed with nonlinear dispersion \cite{OR96}. When applied to the KdV, this approach of tri-Hamiltonian duality yields the Camassa-Holm equation, while \eqref{1} is obtained from the mKdV. Equation \eqref{1} is integrable in the sense that it has a bi-Hamiltonian structure \cite{FU96,OR96,Qiao06} and it admits a Lax pair \cite{Qiao06}. Later, the generalization of \eqref{1} with a dispersive term given below in \eqref{2} was obtained by Qiao \cite{Qiao11} from the two-dimensional Euler equations, with the variable $u$ representing the velocity of the fluid. Qiao also obtained the bi-Hamiltonian structure together with the Lax pair.

In \cite{LLZ24},  Li, Liu, and Zhu investigated the stability of single soliton solutions for the mCH equation with a linear dispersion term, given by
\begin{equation}\label{2}
 m_t + ((u^2 - u_x^2)m)_x + \gamma u_x = 0, \,\,\,\, m = u - u_{xx},
\end{equation}
where the constant $\gamma > 0$ is the linear dispersion parameter. Unlike the mCH equation \eqref{1}, which does not admit smooth traveling wave solutions with vanishing boundary conditions (\cite[Theorem 6.1]{Fu2013} and \cite[Section II.B]{Qiao06}), Equation \eqref{2} admits smooth soliton solutions that decay at infinity \cite{M14, LL21}. By constructing conserved quantities in terms of the momentum variable $m$, it was shown in \cite{LLZ24} that the smooth single soliton of \eqref{2} is orbitally stable under perturbations to $m$ in the Sobolev space $H^3$. In an earlier work \cite{LL21}, orbital stability was established for the same soliton under perturbations to $u$ in $W^{1,4}(\mathbb{R}) \cap H^1(\mathbb{R})$. However, that result applies only to a restricted set of wave speeds \cite[Theorem 1.2]{LL21}. More recently, Li, Liu, and Zhu \cite{LLZ25} extended their analysis to the nonlinear stability of 2-soliton solutions for \eqref{2}. By introducing higher-order conserved quantities in terms of $m$, they showed that the stability of 2-soliton solutions can be established using the constrained minimization framework of \cite{MS93}. This approach has been successfully applied to several other 2-soliton problems, where the dynamics decouple into well-separated solitons over time; see, for example, \cite{HPZ93, KP07, NL06, LW20} and references therein.

As far as stability analysis for the mCH equation \eqref{1} is concerned, the orbital stability of the smooth 1-soliton was established in \cite{DLL25} for perturbations of the momentum variable in $\HT$, using a Lyapunov functional expressed in terms of $m$.
Furthermore, the stability of peakon solutions was demonstrated in \cite{QLL13}.

In addition to one-soliton and one-peakon solutions, the mCH equation \eqref{1} admits several other classes of solutions. On a zero background, it supports multi-peakon solutions \cite{GL13}. On a nonzero background, \eqref{1} admits smooth dark \cite{IL12} and bright \cite{M13} multi-solitons. Solutions to the mCH equation have also been derived via symmetry analysis in \cite{Bies2012}. The bright solitons, which are the focus of this article, can be represented parametrically, similar to the one-soliton solutions \cite{M13}. Furthermore, due to the integrability of the equation, multi-solitons interact elastically during evolution, with no dispersive effects observed at infinity.

Inspired by the works in \cite{HL13, LP22,  EJL24, LLZ25}, we aim to study the stability of 2-soliton solutions of the mCH equation \eqref{1} using a {{Hamiltonian}} structure, written in the variable $m$, and additional conserved quantities. Our general approach is to follow the seminal work presented in \cite{MS93}. More precisely, we first describe the parametrization and asymptotic behavior of the 2-soliton solutions of the mCH equation \eqref{1}. We then demonstrate that such 2-soliton solutions can be considered as critical points of an appropriate action functional expressed in the variable $m$, which is a linear combination of conserved quantities. Next, we aim to determine conditions to ensure that such 2-soliton profiles indeed serve as constrained local minima of the Hessian of the associated action functional. The specific strategy is to show
that such 2-soliton solutions are stable in the constrained space, provided that a constrained {second-order condition \cite{He66}} is satisfied.  Finally, we verify that this constrained second-order condition always holds, indicating that such 2-soliton solutions of the mCH equation \eqref{1} are stable, a consequence of the general stability result established in \cite{MS93}.

Notice that there exist no smooth solitons for the mCH equation \eqref{1} with vanishing boundary condition. Therefore, we consider the solutions of \eqref{1} on a nonzero constant background with $m(t,x)\rightarrow \kappa$ as $x\rightarrow \pm \infty$. Moreover, for fixed $\kappa>0$, we consider the class of functions in the set
\eq{
X_\kappa:=\{m-\kappa\in H^2(\mathbb{R}):m(x)>0,\,x\in \mathbb{R}\}.
}{Xk}
The four conserved integrals we are using to construct the Lyapunov functional are given by
\eq{
 E_1(m)&=\int_{\mathbb{R}}(m-\kappa)\,dx,
\\
E_2(m)&=\int_{\mathbb{R}}\left(\frac{1}{m}-\frac{1}{\kappa}\right)\,dx,
\\E_3(m)&=\int_{\mathbb{R}}\left(\frac{m_x^2}{m^5}+\frac{1}{4m^3}-\frac{1}{4\kappa^3}\right)\,dx,\\
E_4(m)&=\int_{\mathbb{R}}\left(\frac{m_{xx}^2}{2m^7}+\frac{5m_x^2}{4m^7}-\frac{7m_x^4}{2m^9}+\frac{1}{16m^5}-\frac{1}{16\kappa^5}\right)\,dx.
}{7}

The quantity $E_1$ is conserved as a direct consequence of the mCH \eqref{1}, which expresses $m_t$ as the derivative of an expression with respect to $x$. The conserved integrals $E_2$, $E_3$, and $E_4$ were derived in \cite{M13} using a B\"{a}cklund transformation, and the quantities $E_2$ and $E_3$ were identified as two Casimirs for the Hamiltonian structure associated to \eqref{1} in \cite{DLL25}. In addition to that, they were both shown to be conserved by direct computation of their time derivatives. {{In \ref{AA}}}, we explicitly show that $E_4$ is conserved.

It is noted that we used the three conserved functionals $E_1$, $E_2$, and $E_3$ in \cite{DLL25} to establish the following $H^1$-orbital stability of one-soliton of the mCH equation \eqref{1}.

\begin{proposition}\label{pro1.1} For fixed $c>0$, and $\kappa\in\left(\frac{\sqrt{c}}{3}, \frac{\sqrt{3c}}{3}\right)$, there exists a unique smooth solitary wave $m(t, x) = \mu(x - ct)$ of the mCH equation \eqref{1}. This solitary wave $\mu(x - ct)$ is orbitally stable with respect to perturbations in $\HT$, namely, for every $\varepsilon > 0$ there exists $\delta = \delta(\varepsilon) > 0$ such that for every $m_0$ such that $m_0-\kappa \in \HT$ satisfying $\|m_0 - \mu(\cdot)\|_{H^1} < \delta$, there exists a unique solution $m \in C^0(\mathbb{R}, X_\kappa)$ of the mCH equation \eqref{1} with initial datum $m(0, \cdot) = m_0$ and maximal existence time $T > 0$ satisfying
$${
\sup_{t\in [0,T)}\left\{\inf_{r\in \mathbb{R}}\left\{\|m(t,\cdot)-\mu(\cdot+r)\|_{H^1}\right\}\right\}<\varepsilon.}
$$
\end{proposition}

Furthermore, the local well-posedness result and the properties of solutions for the Cauchy problem associated with the mCH
equation \eqref{1} on the line were established in \cite{GL13}.
\begin{proposition} \label{pro1.2}
Let $u_0\in H^s(\mathbb{R})$ with $s>\frac{5}{2}$. Then there exists a time $T>0$ such that the initial value problem \eqref{2.1} has a unique solution $u\in C([0,T),H^s(\mathbb{R}))\cap C^1([0,T),H^{s-1}(\mathbb{R}))$.
Moreover, the solution $u$ depends continuously on the initial data, and if $m_0=(1-\partial_x^2)u_0$ does not change sign, then $m(t,x)$ will not change sign for any $t\in[0,T)$. More precisely, if $m_0(x)>0$, then the corresponding solution $u(t,x)$ is positive and satisfies $|u_x(t,x)|\leq u(t,x)$ for $(t,x)\in [0,T)\times \mathbb{R}$.
\end{proposition}
But more importantly for us is the following result on the global existence of solutions on a nonzero background \cite[Theorem 1.1]{Yang2022}.
\begin{proposition} \label{pro1.3}
Assume that the initial data $m(0) > 0$ and $m(0) -\kappa \in H^{2,1}(\R)\cap  H^{1,2}(\R)$ for some $\kappa>0$. Then there
exists a unique solution of the mCH \eqref{1} such that $m(t)-\kappa \in C\left([0, +\infty);H^{2,1}(\R)\cap  H^{1,2}(\R)\right)$.
\end{proposition}
{\noindent}Here, by $H^{p,s}(\R)$, we mean \cite{Yang2022}
\eq{
H^{p,s}(\R)=\left\{f\in\LT \,\vert\, (1+x^2)^{s/2}\partial_x^jf(x)\in \LT,\;j=0,1,...,p\right\}.
}{Hps}

In \cite{M13}, Matsuno presented a compact parametric representation of the smooth $N$-soliton solution of the mCH equation \eqref{1}. Also, the properties of the one- and two-soliton solutions as well as the general
multi-soliton solutions were investigated in detail. It was demonstrated that one-soliton solutions with speed $c$ exists if and only if $3\kappa^2<c<9\kappa^2$. Moreover, {{Matsuno}} showed that, in the limit $t\rightarrow \infty$,
the asymptotic form of the $N$-solitons in the mCH equation \eqref{1} is a superposition of $N$ decoupled one-solitons.
 We shall denote a general $N$-soliton by
{$U^{(N)}(t,x; c_i,y_{i0})$}, where {each $N$-soliton profile}
 depends upon $N$ wave speeds $c_1, \ldots, c_N$, and $N$ phases $y_{10},\ldots , y_{N0}$.
 In particular, we denote {$U^{(1)} = U^{(1)}(t, x; c, y_{10})$ }as $c$-speed 1-soliton solution
 and $\phi=\phi_c$ only for the $c$-speed 1-soliton profile,
 so that {$U^{(1)} = U^{(1)}(t, x; c, y_{10})=\phi_c(x-ct-y_{10})$}. Furthermore, we denote $\mu_c=\phi_c-\partial_x^2\phi_c$ and $\tilde{\mu}=U^{(N)}-U_{xx}^{(N)}$. Thus, it was shown in \cite{M13} that $U^{(N)}$ can be represented as a sum of $N$ solitons with the following asymptotic behavior
 \eqnn{
 U^{(N)}\sim\sum_{i=1}^{N}\phi_{c_i}(x-c_it-{y_{i0}}) \quad as \quad t\rightarrow \infty,
}
where the phase parameter {$y_{i0}$} depends on the corresponding wave speeds $c_i$.

In this paper, we focus on the nonlinear stability of 2-soliton solution $U^{(2)}$
$(x,t;c_1,c_2,y_{10},y_{20})$ with speeds $3\kappa^2<c_2<c_1<9\kappa^2$, where $y_{10},y_{20}$ are phase parameters. The main result of this article is given in the following theorem. An assumption we make
is that Proposition \ref{pro1.2} can be extended to initial conditions with nonzero background thus for example on $X_\kappa$ as defined in \eqref{Xk}.  Under that assumption, we have the theorem, as we are going to show,
follows from the results of \cite{MS93}.

\begin{theorem}\label{th1.1}    The 2-soliton solution $\tilde{\mu}(t,x; c_1,c_2, y_{10}, y_{20})$ with speeds $3\kappa^2<c_2<c_1<9\kappa^2$ is stable in $X_\kappa$, namely, if for every $\varepsilon>0$ there exists $\delta=\delta(\varepsilon)>0$ such
that for every $m_0\in X_\kappa$ satisfying $\|m_0(\cdot)-\tilde{\mu}(0,\cdot;c_1,c_2, y_{10}, y_{20})\|_{H^2}<\delta$,
then {there exist} $y_{10}(t), y_{20}(t)\in \mathbb{R}$, such that $y_{10}(0)=y_{10}, y_{20}(0)=y_{20}$, and
the corresponding solution $m$ of the mCH equation \eqref{1} with the initial datum
$m(0, \cdot) = m_0$ exists up to the maximal existence time $T>0$ satisfying
\eq{
\sup_{t\in [0,T)}\left\{\|m(t,\cdot)-\tilde{\mu}(t,\cdot; c_1,c_2, y_{10}(t), y_{20}(t))\|_{H^2}\right\}<\varepsilon.
}{stabres}
\end{theorem}

 \begin{remark} \label{re1.0}
We apply the global existence result from Proposition \ref{pro1.3} to reformulate Theorem \ref{th1.1}.
Specifically, we impose the additional conditions that {\( m_0 - \kappa \in H^{2,1}(\mathbb{R}) \cap H^{1,2}(\mathbb{R})
\subset H^2(\mathbb{R}) \) }and that \( m_0 > 0 \), where {\( H^{p,s}(\mathbb{R}) \)} is defined in \eqref{Hps}.
Under these assumptions, we may replace \( T \) with \( \infty \) in \eqref{stabres}.\end{remark}

 \begin{remark} \label{re1.1}
Note that $u=(1-\partial_{xx})^{-1}m$ and
$$
 \|m\|_{H^2}=\int_{\mathbb{R}}(1+x^2)^2\cdot|\hat{u}|\,dx,
$$
which is equivalent to the $H^4(\mathbb{R})$ norm on $u$. Thus it follows from Theorem \ref{th1.1} that if $u(t, x) =U^{(2)}(x,t;c_1,c_2,y_{10},y_{20})$ is the
2-soliton solution of Equation \eqref{1} with {$U^{(2)}\in Y_\kappa$} given by
{\eqnn{
Y_\kappa:=\{u-\kappa\in H^4(\mathbb{R}):u-u_{xx}>0, x\in \mathbb{R}\}
}}
then such 2-soliton solution is stable in {$Y_\kappa$} in the $H^4(\mathbb{R})$ norm.
\end{remark}

\begin{remark} \label{re1.2}
 The 2-soliton solutions studied in \cite{LLZ24} for Equation \eqref{2} are on a zero background, and thus are not related to 2-soliton solutions with non-zero background for the mCH equation \eqref{2} with $\gamma=0$ (i.e. equation \eqref{1}). Hence, our results  are not obtainable by simply taking the limit as $\gamma\to 0$.
 \end{remark}

The remainder of this paper is organized as follows. In Section \ref{2s}, we recall the parametrization and asymptotic behavior of 2-soliton solutions that will later be used in our stability proof.  In Section \ref{3s}, we show that such 2-soliton solutions can be viewed as critical points of an explicit action functional $F$. In Section \ref{4s} we provide analytic spectral properties of the operator $\mathcal{L}=\frac{\delta^2 F}{\delta^2 m}$ evaluated at the 2-soliton solutions, especially that the operator $\mathcal{L}$ has only one negative eigenvalue. Section \ref{5s} is devoted to the proof that the so-called constrained second-order condition is always satisfied. Section \ref{6s} is devoted to the analysis of stability as we complete the proof of the main result in Theorem \ref{th1.1}. {{Also, in \ref{AA}, we demonstrate that the integral $E_4$ defined \eqref{7} is conserved over time.}}

\section {Preliminary results}
\label{2s}
In this section, we first recall some preliminary results, which include the soliton parametrization results and asymptotic behavior of 2-soliton solutions for the mCH equation \eqref{1}.  It is assumed that the solution $m$ satisfies the properties necessary for the existence of the conserved integrals {{given in \eqref{7}}}.

\subsection {Soliton parametrization}\label{2.1s}
As described in \cite{M13}, one can first introduce the reciprocal transformation $(x, t)\rightarrow(y, \tau)$ by
\begin{equation}\label{2.1}
 dy=mdx-m(u^2-u_x^2)dt,\quad d\tau=dt,
\end{equation}
subjected to the restriction $m > 0$. By the application of this transformation, the mCH equation \eqref{1} is recast into the form
\begin{equation}\label{2.2}
m_{\tau}+2m^3u_y=0.
\end{equation}
Notice that from the relation $u_{xx}=m^2u_{yy}+m m_y u_y$, one can obtain the following expression for $u(y,\tau)$
\begin{equation}\label{2.3}
u=m+u_{xx}=m+\frac{1}{2}m\left(\frac{1}{m}\right)_{\tau y}.
\end{equation}
Define a new variable $p(y,\tau)$ by
\eqnn{
p=\frac{1}{m}
}
and substitute it into \eqref{2.2}, one obtains
the following equation for $p(y,\tau)$
\begin{equation}\label{2.5}
pp_{\tau yy}-p_yp_{\tau y}-p^3p_\tau-2p_y=0,
\end{equation}
which is called the associated mCH equation.

By using a B\"acklund transformation between
solutions of the associated mCH equation \eqref{2.5} and a model equation for shallow-water
waves, which is a member of the KdV hierarchy, expressions for $p$ and $u$ are found to be \cite{M13}
\begin{equation}\label{2.6}
p=p(y,\tau)=\frac{1}{\kappa}+2\left(\ln\!\left( \frac{g}{f}\right)\right)_{y}, \quad u=u(y,\tau)=\kappa-\left(\ln\!\left( fg\right)\right)_{\tau y},
\end{equation}
\begin{equation}\label{2.7}
x=x(y,\tau)=\frac{y}{\kappa}+\kappa^2 \tau+2\ln\!\left(\frac{g}{f}\right)+d,
\end{equation}
where $f$ and $g$ represent tau-functions satisfying a bilinear equation {related} to KdV and given in \cite[Equation (2.20)]{M13}, $d$ is an arbitrary constant. We now consider the following two cases.

(i) 1-soliton solution $U^{(1)}$.

The tau-functions $f$ and $g$ are given by
\eqnn{
f=1+e^{\xi}, \quad g=1+e^{\xi-\psi}
}
with
\eqnn{
\xi=k(y-\tilde{c}\tau+y_0),\quad \tilde{c}=\frac{2\kappa^3}{1-(\kappa k)^2},
}
\eqnn{
e^{-\psi}=\frac{1-\kappa k}{1+\kappa k}.
}
The parametric representation of the one-soliton solution can be written in the form
\begin{eqnarray*}
U^{(1)}&=& \kappa-\frac{(fg)(fg)_{\tau y}-(fg)_\tau(fg)_y}{(fg)^2}\\
&=&  \kappa+\frac{k^2\tilde{c}\sqrt{1-(\kappa k)^2}(\cosh (\xi-\xi_0)+\sqrt{1-(\kappa k)^2})}{\left(1+\sqrt{1-(\kappa k)^2}\cosh (\xi-\xi_0)\right)^2},
\end{eqnarray*}
\eqnn{
x=\frac{y}{\kappa}+\kappa^2\tau+2\ln\!\left(\frac{1-\kappa k\tanh\!\left(\frac{\xi}{2}\right)}{1+\kappa k}\right)+d,\quad \xi_0=\frac{1}{2}\ln\!\left(\frac{1+\kappa k}{1-\kappa k}\right).
}

(ii) 2-soliton solution $U^{(2)}$.

The tau-functions $f$ and $g$ are given by
\begin{equation}\label{2.8}
f=1+e^{\xi_1}+e^{\xi_2}+\left(\frac{k_1-k_2}{k_1+k_2}\right)^2e^{\xi_1+\xi_2},
\end{equation}
\begin{equation}\label{2.9}
g=1+e^{\xi_1-\psi_1}+e^{\xi_2-\psi_2}+\left(\frac{k_1-k_2}{k_1+k_2}\right)^2e^{\xi_1+\xi_2-\psi_1-\psi_2},
\end{equation}
where
\begin{equation}\label{2.10}
 \xi_i=k_i(y-\tilde{c_i}\tau+y_{i0}),\quad \tilde{c_i}=\frac{2\kappa^3}{1-(\kappa k_i)^2},\;i=1,2,
\end{equation}
\eqnn{
 e^{-\psi_i}=\frac{1-\kappa k_i}{1+\kappa k_i},\;i=1,2,
}
and $y_{i0}, i = 1, 2$ are arbitrary constants.

\subsection {Asymptotic behavior of 2-soliton solutions}
We discuss briefly the asymptotic behavior of the 2-soliton solutions. Let $c_i,\;i = 1, 2$ be the velocity of the $i$-th soliton in the $(x, t)$ coordinate system and assume that $0 < c_2
< c_1$ and impose the conditions $0<\kappa k_i<\frac{\sqrt{3}}{2}, i=1,2$. Then it follows from \eqref{2.7} that
\eqnn{
 c_i=\frac{\tilde{c_i}}{\kappa}+\kappa^2,\;i=1,2,
}
which implies $k_1>k_2>0$ due to \eqref{2.10}.

We first assume that $\xi_2=O(1)$. When $t\rightarrow +\infty$, in this
case $e^{\xi_1}$ will become exponentially small. Thus it follows from \eqref{2.6}, \eqref{2.7}, \eqref{2.8}, and \eqref{2.9} that
\begin{equation}\label{2.13}
U^{(2)}=\kappa+\frac{k_2^2\tilde{c_2}\sqrt{1-(\kappa k_2)^2}(\cosh (\xi_2-\xi_{20})+\sqrt{1-(\kappa k_2)^2})}{\left(1+\sqrt{1-(\kappa k_2)^2}\cosh (\xi_2-\xi_{20})\right)^2}+e.s.t.,
\end{equation}
 \begin{equation}\label{2.14}
x-c_2t=\frac{\xi_2}{\kappa k_2}+2\ln\!\left(\frac{1-\kappa k_2\tanh\left(\frac{\xi_2}{2}\right)}{1+\kappa k_2}\right) +e.s.t.,
\end{equation}
where
 \eqnn{
\xi_{20}=\frac{1}{2}\ln\!\left(\frac{1+\kappa k_2}{1-\kappa k_2}\right),\quad c_2=\frac{\tilde{c_2}}{\kappa}+\kappa^2=\frac{\kappa^2[3-(\kappa k_2)^2]}{1-(\kappa k_2)^2},
}
and $e.s.t.$ means exponentially small terms that rapidly approach zero.

Similarly, assume that $\xi_1=O(1)$.  As $t\rightarrow +\infty$, $e^{-\xi_2}$ will become exponentially small. The expressions corresponding to \eqref{2.13}--\eqref{2.14} read
\eqnn{
U^{(2)}=\kappa+\frac{k_1^2\tilde{c_1}\sqrt{1-(\kappa k_1)^2}(\cosh (\xi_1-\xi_{10}+\delta_1^{+})+\sqrt{1-(\kappa k_1)^2})}{\left(1+\sqrt{1-(\kappa k_1)^2}\cosh (\xi_1-\xi_{10}+\delta_1^{+})\right)^2}+e.s.t.,
}
 \eqnn{
x-c_1t=\frac{\xi_1}{\kappa k_1}+2\ln\!\left(\frac{1-\kappa k_1\tanh\left(\frac{\xi_1+\delta_1^{+}}{2}\right)}{1+\kappa k_1}\right)-2\psi_2 +e.s.t.,
}
where
 \eqnn{
\xi_{10}=\frac{1}{2}\ln\!\left(\frac{1+\kappa k_1}{1-\kappa k_1}\right),\quad c_1=\frac{\tilde{c_1}}{\kappa}+\kappa^2=\frac{\kappa^2[3-(\kappa k_1)^2]}{1-(\kappa k_1)^2},
}
\eqnn{
\delta_1^{+}=2\ln\!\left(\frac{ k_1-k_2}{k_1+k_2}\right).
}

In the same way, one can show that, except in these two regions, $U^{(2)}$ becomes
exponentially small as $t\rightarrow +\infty$. Thus, the above analysis shows that the long-time asymptotics of a 2-soliton
solution are well-separated into the sum of two single solitons.

\section {Variational characterization}
\label{3s}
In this section, we show that the 2-soliton profiles can be viewed as critical points of an appropriate action functional. Firstly, by recombining the four conservative quantities $E_1$, $E_2$, $E_3$, and $E_4$, we define the following three new conservative quantities
\eq{
 F_1(m)&\triangleq E_2(m)+\frac{1}{\kappa^2}E_1(m)=\int_{\mathbb{R}}\left(\frac{1}{m}+\frac{1}{\kappa^2}m-\frac{2}{\kappa}\right)\,dx,
\\
F_2(m)&\triangleq E_3(m)+\frac{3}{4\kappa^4}E_1(m)=\int_{\mathbb{R}}\left(\frac{m_x^2}{m^5}+\frac{1}{4m^3}+\frac{3}{4\kappa^4}m-\frac{1}{\kappa^3}\right)\,dx,
\\
F_3(m)&\triangleq E_4(m)+\frac{5}{16\kappa^6}E_1(m)
\\&=\int_{\mathbb{R}}\left(\frac{m_{xx}^2}{2m^7}+\frac{5m_x^2}{4m^7}-\frac{7m_x^4}{2m^9}+\frac{1}{16m^5}+\frac{5}{16\kappa^6}m-\frac{3}{8\kappa^5}\right)\,dx.
}{3.1}
The functionals $F_i$, $i=1,2,3$ above are defined in such a way that they are Fr\'echet differentiable on $X_\kappa$.

Next, we consider the minimization of $F_3(m)$ under the constraints defined by
\eqnn{
F_1(m)=F_1(\tilde{\mu}),\quad F_2(m)=F_2(\tilde{\mu}).
}
The associated Lagrangian $F(m)$ is defined by
\eqnn{
F(m):=F_3(m)+\lambda_1F_1(m)+\lambda_2F_2(m), \ \ m\in X_\kappa,
}
where $m=u-u_{xx}$, and $\lambda_1$ and $\lambda_2$ represent two Lagrange multipliers.
For convenience, we express the gradient of the related Lagrangian as
\eqnn{
G(m):=G_3(m)+\lambda_1G_1(m)+\lambda_2G_2(m).
}
The following lemma provides a non-isolated set of critical
points.

\begin{lemma} \label{le3.1}
(Variational characterization) For any fixed $0<c_2<c_1$ and $c_1, c_2\in(3\kappa^2,9\kappa^2), \kappa>0$, the critical points $\tilde{\mu}$ of
Lagrangian $F(m)$ coincide with the 2-soliton profiles $\tilde{\mu}= U^{(2)}-U^{(2)}_{xx}$ for the mCH equation \eqref{1},
if and only if
\eq{
\lambda_1&=-\frac{1}{16\kappa^4}+\frac{1}{2\kappa^2}\left(\frac{1}{c_1-\kappa^2}+\frac{1}{c_2-\kappa^2}\right)+\frac{2}{(c_1-\kappa^2)(c_2-\kappa^2)},\\
 \lambda_2&=-\left(\frac{1}{4\kappa^2}+\frac{1}{c_1-\kappa^2}+\frac{1}{c_2-\kappa^2}\right).
}{3.7}
\end{lemma}

\begin{proof}
The proof is divided into three steps.

{\sl Step 1. Computation of the 1-soliton relation.}

Recall that the 1-soliton can be determined by searching for a traveling wave solution of \eqref{1} of the form $u(t,x)=\phi_c(\xi)$, $\xi=x-ct$.
After integration, we find that the profile $\phi_c$ satisfies
\begin{equation}\label{3.9}
(\phi_c-\partial_x^2\phi_c)((\partial_x\phi_c)^2-\phi^2+c)=A,
\end{equation}
where $A$ is an integration constant. It is easily seen that \eqref{3.9} has the first integral
\begin{equation}\label{3.10}
(\phi_c^2-(\partial_x\phi_c)^2-c)^2+4A\phi_c=E,
\end{equation}
where $E$ is another integral constant.

The smooth solitary wave corresponds to the homoclinic orbit to the
equilibrium point $(\phi_c, \phi_{c,x} ) = (\kappa, 0)$ \cite[Section 3]{DLL25}, {\color{black}{where we have written $\phi_{c,x} =\partial_x\phi_c$ for convenience.}} So the homoclinic orbit must be given by
\begin{equation}\label{3.11}
 \phi_c^2-\phi_{c,x}^2-c=-\sqrt{E-4A\phi_c},
\end{equation}
where
\begin{equation}\label{3.12}
 A=\kappa(c-\kappa^2), \ \ E=(c-\kappa^2)(c+3\kappa^2).
\end{equation}
Using \eqref{3.9} and \eqref{3.11} we have
\begin{equation}\label{3.13}
\mu_c=\phi_c-\phi_{c,xx}=\frac{A}{\phi_{c,x}^2-\phi_c^2+c}=\frac{A}{\sqrt{E-4A\phi_c}}.
\end{equation}
Thus, a direct calculation shows that
\begin{equation}\label{3.14}
 \mu_{c,x}=\frac{2A\phi_{c,x}\mu_c}{(\phi_c^2-\phi_{c,x}^2-c)^2}=2A^{-1}\phi_{c,x}\mu_c^3,
\end{equation}
\eqnn{
\mu_{c,xx}=2A^{-1}(\phi_c-\mu_c)\mu_c^3+12A^{-2}\phi_{c,x}^2\mu_c^5,
}
and
\begin{equation}\label{3.16}
\mu_{c,xxx}=\mu_{c,x}(1-20A^{-1}\mu_c^3+18A^{-1}\mu_c^3\phi_c+60A^{-2}\phi_{c,x}^2\mu_c^4).
\end{equation}

On the other hand, a simple calculation gives
\eqnn{
 G_1(m)&=\frac{1}{\kappa^2}-\frac{1}{m^2},
\\
G_2(m)&=\frac{-2m_{xx}}{m^5}+\frac{5m_x^2}{m^6}-\frac{3}{4m^4}+\frac{3}{4\kappa^4},
}
and
\eqnn{
G_3(m)&=  \frac{m_{xxxx}}{m^7}-\frac{14m_xm_{xxx}}{m^8}-\frac{21m_{xx}^2}{2m^8}+\frac{98m_x^2m_{xx}}{m^9}\\
&   -\frac{189m_x^4}{2m^{10}}-\frac{5m_{xx}}{2m^7}+\frac{35m_x^2}{4m^8}-\frac{5}{16m^6}+\frac{5}{16\kappa^6}.
}

Our aim now is to show that, there are two variational equations
for single solitons,
 \begin{equation}\label{3.20}
G_2(\mu_c)=\omega_1G_1(\mu_c),\quad G_3(\mu_c)=\omega_2G_1(\mu_c),
 \end{equation}
where the constants $\omega_1$ and $\omega_2$ are uniquely determined by $c$ and $\kappa$.

To simplify our computations, we do not calculate $G_i(\mu_c),\;i=1,2,3$ directly, but instead we choose to calculate $G_i(\mu_c)\cdot \mu_{c,x},\;i=1,2,3$.
The reason for this is that we find that \eqref{3.20} is equivalent to
$$
G_2(\mu_c)\cdot \mu_{c,x}=\omega_1G_1(\mu_c)\cdot \mu_{c,x},\quad G_3(\mu_c)\cdot \mu_{c,x}=\omega_2G_1(\mu_c)\cdot \mu_{c,x}
$$
because of continuity and the fact that the zeros of $\mu_{c,x}$ are all isolated.

Note that
 \eqnn{
G_1(m)\cdot m_x&=\left(\frac{1}{\kappa^2}-\frac{1}{m^2}\right)m_x=\left(\frac{1}{\kappa^2}m+\frac{1}{m}\right)_x,
\\
G_2(m)\cdot m_x&= \left(\frac{-2m_{xx}}{m^5}+\frac{5m_x^2}{m^6}-\frac{3}{4m^4}+\frac{3}{4\kappa^4}\right)m_x \nonumber\\
&= \left(\frac{-m_x^2}{m^5}+\frac{1}{4m^3}+\frac{3}{4\kappa^4}m\right)_x,
}
and
\begin{eqnarray*}
G_3(m)\cdot m_x&=&  \left(\frac{m_{xxxx}}{m^7}-\frac{14m_xm_{xxx}}{m^8}-\frac{21m_{xx}^2}{2m^8}+\frac{98m_x^2m_{xx}}{m^9}\right.\nonumber\\
& & \left. -\frac{189m_x^4}{2m^{10}}-\frac{5m_{xx}}{2m^7}+\frac{35m_x^2}{4m^8}-\frac{5}{16m^6}+\frac{5}{16\kappa^6}\right)m_x\\
&=&  \left(\frac{m_xm_{xxx}}{m^7}-\frac{7m_{xx}m_{x}^2}{m^8}-\frac{m_{xx}^2}{2m^7}+\frac{21m_x^4}{2m^9}\right.\nonumber\\
& &  \left. -\frac{5m_x^2}{4m^{7}}+\frac{1}{16m^5}+\frac{5}{16\kappa^6}m\right)_x\nonumber.
\end{eqnarray*}

Recall from \eqref{3.13}, we have
\begin{equation}\label{3.24}
 \phi_c=\frac{E}{4A}-\frac{A}{4\mu_c^2},
 \end{equation}
and
\begin{equation}\label{3.25}
 \phi_{c,x}^2=\frac{A}{\mu_c}+\phi_c^2-c=\frac{A}{\mu_c}-\frac{E}{8\mu_c^2}+\frac{A^2}{16\mu_c^4}+\frac{E^2}{16A^2}-c.
 \end{equation}
Substituting \eqref{3.24} and \eqref{3.25} into \eqref{3.14}, we get
\begin{equation}\label{3.26}
 \mu_{c,x}^2=4A^{-2}\phi_{c,x}^2\mu_c^6=\frac{1}{4}\mu_c^2-\frac{1}{2}A^{-2}E\mu_c^4+4A^{-1}\mu_c^5+\left(\frac{1}{4}A^{-4}E^2-4A^{-2}c\right)\mu_c^6.
 \end{equation}
In view of \eqref{3.12} and \eqref{3.26}, a straight-forward derivation leads to
\begin{eqnarray*}
 \frac{-\mu_{c,x}^2}{\mu_c^5}+\frac{1}{4\mu_c^3}+\frac{3}{4\kappa^4}\mu_c
 &=&  \left(4A^{-2}c-\frac{1}{4}A^{-4}E^2+\frac{3}{4\kappa^4}\right)\mu_c+\frac{E}{2A^2}\cdot \frac{1}{\mu_c}-4A^{-1}\nonumber\\
&=&  \frac{c+3\kappa^2}{2\kappa^4(c-\kappa^2)}\mu_c+\frac{c+3\kappa^2}{2\kappa^2(c-\kappa^2)}\cdot \frac{1}{\mu_c}-\frac{4}{\kappa(c-\kappa^2)}\nonumber\\
&=& \frac{c+3\kappa^2}{2\kappa^2(c-\kappa^2)}\left(\frac{1}{\kappa^2}\mu_c+\frac{1}{\mu_c}\right)-\frac{4}{\kappa(c-\kappa^2)}.
\end{eqnarray*}
This means that
\begin{equation}\label{3.28}
 G_2(\mu_c)\cdot \mu_{c,x}=\frac{c+3\kappa^2}{2\kappa^2(c-\kappa^2)}\left(\frac{1}{\kappa^2}\mu_c+\frac{1}{\mu_c}\right)_x=\frac{c+3\kappa^2}{2\kappa^2(c-\kappa^2)}G_1(\mu_c)\cdot \mu_{c,x}.
 \end{equation}

Similarly, applying \eqref{3.14}--\eqref{3.16} and \eqref{3.24}--\eqref{3.26}, after a simplification, we can obtain
\begin{eqnarray*}
& & \frac{\mu_{c,x}\mu_{c,xxx}}{\mu_c^7}-\frac{7\mu_{c,xx}\mu_{c,x}^2}{\mu_c^8}-\frac{\mu_{c,xx}^2}{2\mu_c^7}+\frac{21\mu_{c,x}^4}{2\mu_c^9} -\frac{5\mu_{c,x}^2}{4\mu_c^{7}}+\frac{1}{16\mu_c^5}+\frac{5}{16\kappa^6}\mu_c \nonumber\\
&=&  -A^{-2}\phi_{c,x}^2\mu_c^{-1}-2A^{-2}(\phi_c-\mu_c)^2\mu_c^{-1}-8A^{-3}\phi_c\mu_c\phi_{c,x}^2+\frac{1}{16\mu_c^5}+\frac{5}{16\kappa^6}\mu_c\nonumber\\
&=& \frac{3c^2+2c\kappa^2+27\kappa^4}{16\kappa^4(c-\kappa^2)^2}\left(\frac{1}{\kappa^2}\mu_c+\frac{1}{\mu_c}\right)-\frac{c+3\kappa^2}{\kappa^3(c-\kappa^2)^2},
\end{eqnarray*}
which means that
\begin{equation}\label{3.30}
 G_3(\mu_c)\cdot \mu_{c,x}=\frac{3c^2+2c\kappa^2+27\kappa^4}{16\kappa^4(c-\kappa^2)^2}G_1(\mu_c)\cdot \mu_{c,x}.
 \end{equation}
Therefore, by \eqref{3.28} and \eqref{3.30}, we have
\eq{
 G_2(\mu_c)=\frac{c+3\kappa^2}{2\kappa^2(c-\kappa^2)}G_1(\mu_c), \quad G_3(\mu_c)=\frac{3c^2+2c\kappa^2+27\kappa^4}{16\kappa^4(c-\kappa^2)^2}G_1(\mu_c).
 }{G3G2G1}\\

{\sl Step 2. Search for candidate multipliers for 2-soliton.}

Assume that the two-soliton solution $\tilde{\mu}(x; c_1, c_2, y_{10}, y_{20})$ is a critical point of
$F$, {\color{black}{that is,}}  the Fr\'echet derivative of $F$ denoted by $G$ is such that $G(\tilde{\mu})=0$. Since $G(\tilde{\mu})\in H^{\infty}(\mathbb{R})$, $G(\tilde{\mu})=0$ implies almost everywhere (a.e.) point-wise vanishing.
Consequently, the long-time asymptotic behavior of 2-soliton ensure that $G(\mu_{c_1})=G(\mu_{c_2})=0$. It means that single soliton profiles $\mu_{c_1}$ and $\mu_{c_2}$ with speeds $c_1$ and $c_2$ respectively,  are also critical points of $F$. This fact can help to determine the multipliers uniquely.

Note that $G(m)=G_3(m)+\lambda_1G_1(m)+\lambda_2G_2(m)$, then $G(\mu_{c_1})=G(\mu_{c_2})=0$ reads
\begin{equation}\label{3.32}
 \left(\lambda_1+\frac{3c_j^2+2c_j\kappa^2+27\kappa^4}{16\kappa^4(c_j-\kappa^2)^2}+\lambda_2\cdot\frac{c_j+3\kappa^2}{2\kappa^2(c_j-\kappa^2)}\right)G_1(\mu_{c_j})=0, \;j=1,2,
 \end{equation}
 where we have used \eqref{G3G2G1}.
Because $G_1(\mu_{c_j})\neq 0,\;i=1,2$, it follows from \eqref{3.32} that
\eqnn{
 \lambda_1+\frac{3c_1^2+2c_1\kappa^2+27\kappa^4}{16\kappa^4(c_1-\kappa^2)^2}+\lambda_2\cdot\frac{c_1+3\kappa^2}{2\kappa^2(c_1-\kappa^2)}=0,\\
 \lambda_1+\frac{3c_2^2+2c_2\kappa^2+27\kappa^4}{16\kappa^4(c_2-\kappa^2)^2}+\lambda_2\cdot\frac{c_2+3\kappa^2}{2\kappa^2(c_2-\kappa^2)}=0.
}
By solving the above system, we obtain $\lambda_1$ and $\lambda_2$ as in \eqref{3.7}.
\\

{\sl Step 3. Verification of $G(\tilde{\mu})= 0$.}

As shown in Section 2, any 2-soliton $\tilde{\mu}$ can be expressed as the sum of individual solitons with exponentially
small errors as $t\rightarrow \infty$. Then this implies that the long-time asymptotic profile will yield
\begin{equation*}
 G(\tilde{\mu})=G(\mu_{c_1})+G(\mu_{c_2})+e.s.t..
 \end{equation*}
Furthermore, it is demonstrated in Proposition 3.3 of \cite{LW20} that, if the 2-soliton $\tilde{\mu}$ {\color{black}{satisfies}} that
\begin{equation*}
 \tilde{\mu}=\mu_{c_1}+\mu_{c_2}+e.s.t., \quad G(\mu_{c_1})=G(\mu_{c_2})=0,
 \end{equation*}
as $t$ is very large, then there exists $\lambda_1$ and $\lambda_2$ such that this 2-soliton $\tilde{\mu}$ verifies the variational principle
\begin{equation*}
 G(\tilde{\mu})=G_3(\tilde{\mu})+\lambda_1G_1(\tilde{\mu})+\lambda_2G_2(\tilde{\mu})=0.
 \end{equation*}
Thus it implies that as shown in Step 2, $\lambda_1$ and $\lambda_2$ must be
given by \eqref{3.7}.
The proof of Lemma \ref{le3.1} is completed.
\end{proof}

\section {Spectral Analysis}
\label{4s}
To establish that the 2-soliton profiles $\tilde{\mu}$ represent constrained local minimizers,
 we need to examine the second variation operator denoted as $\mathcal{L}:=\frac{\delta^2 F}{\delta m^2}(\tilde{\mu})$. By a straight-forward calculation, we have
\begin{eqnarray}\label{4.1}
\mathcal{L}[h]&:=& (\tilde{\mu}^{-7}h_{xx})_{xx}+\left[\left(42\tilde{\mu}^{-9}\tilde{\mu}_x^2-14\tilde{\mu}^{-8}\tilde{\mu}_{xx}-\frac{5}{2}\tilde{\mu}^{-7}-2\lambda_2\tilde{\mu}^{-5}\right)h_x\right]_x\nonumber\\
& &  +\left(\frac{15}{8}\tilde{\mu}^{-7}+2\lambda_1\tilde{\mu}^{-3}+3\lambda_2\tilde{\mu}^{-5}\right)h+q(\tilde{\mu})h,
\end{eqnarray}
where
\begin{eqnarray*}
q(\tilde{\mu})&:=& 28\tilde{\mu}^{-9}\tilde{\mu}_{xx}^2+70\tilde{\mu}^{-9}\tilde{\mu}_x^2-\frac{5}{2}\left(\tilde{\mu}^{-7}\right)_{xx}-315\tilde{\mu}^{-11}\tilde{\mu}_x^4-126\left(\tilde{\mu}^{-10}\tilde{\mu}_x^3\right)_x\\
& &  -7\left(\tilde{\mu}^{-8}\tilde{\mu}_{xx}\right)_{xx}+30\lambda_2\tilde{\mu}^{-7}\tilde{\mu}_x^2-2\lambda_2\left(\tilde{\mu}^{-5}\right)_{xx}.
\end{eqnarray*}
with $q(\tilde{\mu})$ above being a smooth function of the 2-soliton profile $\tilde{\mu}$, rapidly
decaying to $0$ as $|x|\rightarrow \infty$. In this section, we describe the spectrum of this operator. Our analysis follows the arguments presented in \cite{Lax68, MS93}.

\begin{lemma} \label{le4.1}
 $\mathcal{L}$ is a linear, unbounded self-adjoint operator in $L^2(\mathbb{R})$, with dense domain $H^4(\mathbb{R})$.
\end{lemma}

\begin{proof}
Let $h,w\in H^4(\mathbb{R})$. Integrating by parts, one has
\begin{eqnarray*}
\int_\mathbb{R}w\mathcal{L}[h]&:=&\int_\mathbb{R}w(\tilde{\mu}^{-7}h_{xx})_{xx}+\int_{\mathbb{R}}w[(42\tilde{\mu}^{-9}\tilde{\mu}_x^2-14\tilde{\mu}^{-8}\tilde{\mu}_{xx}-\frac{5}{2}\tilde{\mu}^{-7}-2\lambda_2\tilde{\mu}^{-5})h_x]_x\\
& &  +\int_{\mathbb{R}}\left(\frac{15}{8}\tilde{\mu}^{-7}+2\lambda_1\tilde{\mu}^{-3}+3\lambda_2\tilde{\mu}^{-5}\right)wh+q(\tilde{\mu})wh\\
&=& \int_\mathbb{R}(\tilde{\mu}^{-7}w_{xx})_{xx}h+\int_{\mathbb{R}}[(42\tilde{\mu}^{-9}\tilde{\mu}_x^2-14\tilde{\mu}^{-8}\tilde{\mu}_{xx}-\frac{5}{2}\tilde{\mu}^{-7}-2\lambda_2\tilde{\mu}^{-5})w_x]_xh\\
& &  +\int_{\mathbb{R}}\left(\frac{15}{8}\tilde{\mu}^{-7}+2\lambda_1\tilde{\mu}^{-3}+3\lambda_2\tilde{\mu}^{-5}\right)wh+q(\tilde{\mu})wh=\int_{\mathbb{R}}\mathcal{L}[w]h.
\end{eqnarray*}
Moreover, it is clear that $D(\mathcal{L}^*)$ can be identified with $D(\mathcal{L})=H^4(\mathbb{R})$.
\end{proof}
A consequence of Lemma \ref{le4.1} is the fact that the spectrum of $\mathcal{L}$ is real-valued. The following result describes the essential spectrum of $\mathcal{L}$.

\begin{lemma} \label{le4.2}
The essential spectrum of the operator $\mathcal{L}$ is equal to the spectrum of its corresponding asymptotic  constant
coefficients operator
\begin{eqnarray*}
\mathcal{L}_{\infty}[h]&:=& \kappa^{-7}h_{xxxx}-\left(\frac{5}{2}\kappa^{-7}+2\lambda_2\kappa^{-5}\right)h_{xx}+\left(\frac{15}{8}\kappa^{-7}+2\lambda_1\kappa^{-3}+3\lambda_2\lambda^{-5}\right)h\nonumber\\
&=& \kappa^{-7}\left(-D_x^2+\frac{c_1-3\kappa^2}{c_1-\kappa^2}\right)\left(-D_x^2+\frac{c_2-3\kappa^2}{c_2-\kappa^2}\right)h.
\end{eqnarray*}
In particular, the essential spectrum of $\mathcal{L}$ is $\sigma_{ess}(\mathcal{L})=\left[\frac{(c_1-3\kappa^2)(c_2-3\kappa^2)}{\kappa^7(c_1-\kappa^2)(c_2-\kappa^2)},+\infty\right)$.
\end{lemma}

\begin{proof}
 The fact that the essential spectrum  of $\mathcal{L}$ is equal to the constant coefficient operator $\mathcal{L}_{\infty}$ is a classical result obtained from applying Weyl's essential spectrum theorem (see for example \cite[Theorem A.2, p.\ 140]{Henry81}).  The spectrum {of the} constant coefficient operator $\mathcal{L}_{\infty}$ is obtained
by Fourier analysis, that is by finding the values of $\lambda$ for which the equation
$$
\left(\mathcal{L}_{\infty}-\lambda\right)[e^{\ri\sigma x}]=0
$$
has a solution for $\sigma\in \R$. It is then a straightforward computation to find that the essential spectrum is given by the interval specified in the lemma.
\end{proof}
\begin{remark}
Note that $\frac{c_1-3\kappa^2}{c_1-\kappa^2}>0$ and $\frac{c_2-3\kappa^2}{c_2-\kappa^2}>0$. The essential spectrum of $\mathcal{L}$ specified by Lemma \ref{le4.2} thus is positive.
\end{remark}

We introduce now two directions associated to spatial translations. We denote
\eqnn{
 \tilde{\mu}(t;y_{10},y_{20}):=\tilde{\mu}(t,x;c_1,c_2,y_{10},y_{20}),
 }
and define
\begin{equation}\label{4.3}
 \mu_1(t;y_{10},y_{20}):=\partial_{y_{10}}\tilde{\mu}(t;y_{10},y_{20}),\quad  \mu_2(t;y_{10},y_{20}):=\partial_{y_{20}}\tilde{\mu}(t;y_{10},y_{20}).
 \end{equation}
The quantities $\mu_1$ and $\mu_2$ are linearly independent as functions of the $x$-variable, for all
time $t$ fixed, as can be seen from the expressions given in \eqref{2.8} and \eqref{2.9} in the case $k_1\neq k_2$.

\begin{lemma} \label{le4.3}
For each $t\in \mathbb{R}$, one has
\begin{equation*}
\ker(\mathcal{L})=\Span\left\{\mu_1(t;y_{10},y_{20}),\mu_2(t;y_{10},y_{20})\right\}.
\end{equation*}
\end{lemma}

\begin{proof}
It follows from Lemma \ref{le3.1} that $\partial_{y_{10}}G(\tilde{\mu})=\partial_{y_{20}}G(\tilde{\mu})=0$. Writing down
these identities, we obtain
\begin{equation}\label{4.4}
\mathcal{L}[\mu_1](t;y_{10},y_{20})=\mathcal{L}[\mu_2](t;y_{10},y_{20})=0,
\end{equation}
with $\mathcal{L}$ the operator defined in \eqref{4.1} and $\mu_1, \mu_2$ defined in \eqref{4.3}.
On the other hand, $\mathcal{L}_{\infty}h=0$ has four pure exponential solutions $e^{\pm\sqrt{\frac{c_1-3\kappa^2}{c_1-\kappa^2}}x}, e^{\pm\sqrt{\frac{c_2-3\kappa^2}{c_2-\kappa^2}}x}$.
Consequently, it means that 2 solutions decay as $x\rightarrow +\infty$, and 2 solutions decay as $x\rightarrow -\infty$. Therefore, the null space $\ker(\mathcal{L})$ is spanned by at most two linearly independent solutions in $L^2(\mathbb{R})$.
Finally, in view of \eqref{4.4}, we have the desired conclusion.
\end{proof}

Next, we will give the most important result of this section, which is that the operator $\mathcal{L}$ has only one negative eigenvalue. In order to prove that result, we
follow the strategy presented in \cite{Green91, MS93}. More specifically, we need to use the following conclusions.

\begin{lemma} \label{le4.4} \cite[Lemma 2.2]{MS93} \cite[Section 5]{Green91}
Let $\tilde{\mu}$ be any 2-soliton solution, and $\mu_1, \mu_2$ the corresponding kernel of the operator $\mathcal{L}$. Then $\mathcal{L}$ has
precisely $r$ negative eigenvalues where $r$ is number of roots, counted with appropriate
multiplicity, of the second order Wronskian determinant $Wr(x)$, where
\begin{equation*}
Wr(x):=Wr_2(\mu_1, \mu_2)=\left | \begin{matrix}
\mu_1& \mu_2 \\
(\mu_1)_x& (\mu_2)_x \\
\end{matrix} \right |.
\end{equation*}
\end{lemma}

The following proposition asserts that
$\mathcal{L}$ has precisely one negative eigenvalue.

\begin{proposition} \label{pro4.1}
 The operator $\mathcal{L}$ has exactly one negative eigenvalue for all finite values
of $y_{10}, y_{20}$ and all $c_1, c_2$ with $3\kappa^2 < c_2 < c_1<9\kappa^2$.
\end{proposition}

\begin{proof}
By lemmas \ref{le4.3} and
 \ref{le4.4}, it suffices to check the number of zeroes of the Wronskian determinant $Wr(x)$. However, it is very challenging to calculate $\mu_j=\frac{\partial \tilde{\mu}}{\partial y_{j0}}=\partial_{y_{j0}} (U^{(2)}-U^{(2)}_{xx})$ directly from the parameterization in (2.6) and (2.7). Thus, we seek a more convenient parametrization for $\frac{\partial \tilde{\mu}}{\partial y_{j0}}$.

Recall from (2.4) and (2.6) that $\tilde{\mu}=\frac{1}{p},\quad p=\frac{1}{\kappa}+2\left(\ln\!\left(\frac{g}{f}\right)\right)_y$. By using the parameterization outlined in Section \ref{2.1}, we can obtain
\begin{equation*}
 \mu_1=\frac{\partial \tilde{\mu}}{\partial y_{10}}=k_1\frac{\partial \tilde{\mu}}{\partial \xi_1}=-\frac{k_1}{p^2}p_{\xi_1},
\end{equation*}
and
\begin{equation*}
 \mu_2=\frac{\partial \tilde{\mu}}{\partial y_{20}}=k_2\frac{\partial \tilde{\mu}}{\partial \xi_2}=-\frac{k_2}{p^2}p_{\xi_2}.
\end{equation*}
Consequently, {{in view of $\frac{\partial x}{\partial y}=m$ (see (2.1))}}, we have
{\eqnn{
Wr(x)&:= \mu_1(\mu_2)_{x}-\mu_2(\mu_1)_x=\frac{k_1k_2}{mp^4}(p_{\xi_1}p_{\xi_2y}-p_{\xi_2}p_{\xi_1y}) \nonumber\\
&= \frac{4k_1k_2}{mp^4}\left[\left(\ln\!\left(\frac{g}{f}\right)\right)_{y\xi_1}\left(\ln\!\left(\frac{g}{f}\right)\right)_{y\xi_2y}-\left(\ln\!\left(\frac{g}{f}\right)\right)_{y\xi_2}\left(\ln\!\left(\frac{g}{f}\right)\right)_{y\xi_1y}\right].
}}
Thus we only need to compute the number of zeros of the following Wronskian
{\begin{equation}\label{4.6}
\tilde{W}r(y)\triangleq \left(\ln\!\left(\frac{g}{f}\right)\right)_{y\xi_1}\left(\ln\!\left(\frac{g}{f}\right)\right)_{y\xi_2y}-\left(\ln\!\left(\frac{g}{f}\right)\right)_{y\xi_2}\left(\ln\!\left(\frac{g}{f}\right)\right)_{y\xi_1y}.
\end{equation}}
Recall the explicit formula for the two tau-functions
\eqnn{
f&=1+e^{\xi_1}+e^{\xi_2}+e^{\xi_1+\xi_2+2h},\\
g&=1+e^{\xi_1-\psi_1}+e^{\xi_2-\psi_2}+e^{\xi_1+\xi_2-\psi_1-\psi_2+2h},
}
where
\begin{equation}\label{4.9}
 \xi_i=k_i(y-\tilde{c_i}\tau+y_{i0}),\quad \tilde{c_i}=\frac{2\kappa^3}{1-(\kappa k_i)^2},\;i=1,2,
\end{equation}
\begin{equation}\label{4.10}
 e^{-\psi_i}=\frac{1-\kappa k_i}{1+\kappa k_i},\;i=1,2,\quad e^{2h}=\left(\frac{k_1-k_2}{k_1+k_2}\right)^2.
\end{equation}

Now, we shall apply the large time asymptotics of the 2-soliton to simplify the
required calculations. In fact, as shown in Section \ref{2s}, $U^{(2)}$ is exponentially small except in the regions where one
 of the quantities $(y-\tilde{c}_i\tau)=O(1),\;i=1,2$. In other words, one can find that for large times $\tau$, the 2-soliton is closely approximated
by the sum of two single solitons each centered on regions which are far apart. Moreover, in each such regions only two terms in the tau-function $f$ and $g$ are of importance.
This helps us simplify the calculation of the Wronskian $\tilde{W}r(y)$. And then we start to compute the desired Wronskian in the asymptotic regions. \\

{\sl Case I: Assume that $\tau\gg 1$ and $(y-\tilde{c}_2\tau)=O(1)$.}
In this case, it follows from $\tilde{c}_1>\tilde{c}_2$ that $\xi_1=k_1(y-\tilde{c_1}\tau+y_{10})\rightarrow -\infty$. This
implies that $e^{\xi_1}\ll 1$, allowing us to omit many terms in \eqref{4.6}. In particular, we find that up to exponentially small terms
\begin{eqnarray}\label{4.11}
\left(\ln\!\left(\frac{g}{f}\right)\right)_{y\xi_1}&=& \left(\frac{(1+e^{\xi_2})(e^{\xi_1-\psi_1}+e^{\xi_1+\xi_2-\psi_1-\psi_2+2h})-(1+e^{\xi_2-\psi_2})(e^{\xi_1}+e^{\xi_1+\xi_2+2h})}{(1+e^{\xi_2})(1+e^{\xi_2-\psi_2})}\right)_y \nonumber\\
& \triangleq & (f_1)_y,
\end{eqnarray}
\begin{equation}\label{4.12}
\left(\ln\!\left(\frac{g}{f}\right)\right)_{y\xi_2}=\left(\frac{(e^{-\psi_2}-1)e^{\xi_2}}{(1+e^{\xi_2})(1+e^{\xi_2-\psi_2})}\right)_y\triangleq(f_2)_y.
\end{equation}
Thus, it follows from \eqref{4.6}, \eqref{4.11} and \eqref{4.12} that
\begin{equation}\label{4.13}
\tilde{W}r(y)=Wr_2((f_1)_y,(f_2)_y)+e.s.t..
\end{equation}
So we just need to estimate the number of zeros of second order Wronskian determinant $Wr_2((f_1)_y,(f_2)_y)$.

Note that the denominators of these two functions $f_1$ and $f_2$ are exactly the same, and then by the homogeneity property of Wronskians, we obtain
\begin{eqnarray}\label{4.14}
 Wr_2((f_1)_y,(f_2)_y)&:=&  Wr_3(1,f_1,f_2)\\
&=& (1+e^{\xi_2})^{-3}(1+e^{\xi_2-\psi_2})^{-3}Wr_3\left((1+e^{\xi_2})(1+e^{\xi_2-\psi_2}),g_1,g_2\right)\nonumber,
\end{eqnarray}
where
\begin{eqnarray*}
 g_1&:=&  (1+e^{\xi_2})(e^{\xi_1-\psi_1}+e^{\xi_1+\xi_2-\psi_1-\psi_2+2h})-(1+e^{\xi_2-\psi_2})(e^{\xi_1}+e^{\xi_1+\xi_2+2h})\nonumber\\
&=& (e^{-\psi_1}-1)(1+e^{2\xi_2-\psi_2+2h})e^{\xi_1}+(e^{-\psi_1}+e^{-\psi_1-\psi_2+2h}-e^{2h}-e^{-\psi_2})e^{\xi_1+\xi_2}\nonumber,\\
& &g_2:=(e^{-\psi_2}-1)e^{\xi_2}.
\end{eqnarray*}
 For the sake of convenience, we denote
\begin{equation}\label{4.16}
 \alpha_i:=e^{-\psi_i}-1,\;i=1,2,\quad \beta:=e^{-\psi_1}+e^{-\psi_1-\psi_2+2h}-e^{2h}-e^{-\psi_2}.
\end{equation}
Noting that $g_2=\alpha_2 e^{\xi_2}$, we find after some manipulations
\begin{eqnarray}\label{4.17}
& & Wr_3\left((1+e^{\xi_2})(1+e^{\xi_2-\psi_2}),g_1,g_2\right)\nonumber \\
&=& \alpha_2[\alpha_1Wr_3\left(1+e^{2\xi_2-\psi_2}, (1+e^{2\xi_2-\psi_2+2h})e^{\xi_1},e^{\xi_2}\right)\nonumber\\
& & +\beta Wr_3\left(1+e^{2\xi_2-\psi_2}, e^{\xi_1+\xi_2},e^{\xi_2}\right)]\\
&\triangleq & \alpha_2[\alpha_1J_1+\beta J_2].\nonumber
\end{eqnarray}
Recall that $e^h=\frac{k_1-k_2}{k_1+k_2}$, by using properties of determinants, one has
\begin{eqnarray*}
 J_1&=& Wr_3\left(1+e^{2\xi_2-\psi_2}, (1+e^{2\xi_2-\psi_2+2h})e^{\xi_1},e^{\xi_2}\right)\nonumber\\
&=& \left | \begin{matrix}
1+e^{2\xi_2-\psi_2}& (1+e^{2\xi_2-\psi_2+2h})e^{\xi_1} & e^{\xi_2}\\
2k_2e^{2\xi_2-\psi_2} & [k_1+(k_1+2k_2)e^{2\xi_2-\psi_2+2h}]e^{\xi_1} & k_2e^{\xi_2} \\
4k_2^2e^{2\xi_2-\psi_2} & [k_1^2+(k_1+2k_2)^2e^{2\xi_2-\psi_2+2h}]e^{\xi_1} & k_2^2e^{\xi_2} \\
\end{matrix} \right |\nonumber\\
&=& \left | \begin{matrix}
1+e^{2\xi_2-\psi_2}& (1+e^{2\xi_2-\psi_2+2h})e^{\xi_1} & e^{\xi_2}\\
k_2e^{2\xi_2-\psi_2}-k_2 & [k_1-k_2+(k_1+k_2)e^{2\xi_2-\psi_2+2h}]e^{\xi_1} & 0 \\
2k_2^2e^{2\xi_2-\psi_2} & [k_1(k_1-k_2)+(k_1+k_2)(k_1+2k_2)^2e^{2\xi_2-\psi_2+2h}]e^{\xi_1} & 0 \\
\end{matrix} \right |\nonumber\\
&=& \frac{k_1k_2(k_1-k_2)}{k_1+k_2}e^{\xi_1+\xi_2}[(k_1-k_2)e^{4\xi_2-2\psi_2}-2k_2e^{2\xi_2-\psi_2}-(k_1+k_2)]\\
&=& k_1k_2(k_1+k_2)e^{\xi_1+\xi_2+h}(e^{2\xi_2-\psi_2}+1)(e^{2\xi_2-\psi_2+h}-1).
\end{eqnarray*}
Similarly
\begin{eqnarray*}
 J_2 &=& Wr_3\left(1+e^{2\xi_2-\psi_2}, e^{\xi_1+\xi_2},e^{\xi_2}\right)\nonumber\\
&=& \left | \begin{matrix}
1+e^{2\xi_2-\psi_2}& e^{\xi_1+\xi_2} & e^{\xi_2}\\
2k_2e^{2\xi_2-\psi_2} & (k_1+k_2)e^{\xi_1+\xi_2} & k_2e^{\xi_2} \\
4k_2^2e^{2\xi_2-\psi_2} & (k_1+k_2)^2e^{\xi_1+\xi_2} & k_2^2e^{\xi_2} \\
\end{matrix} \right |\nonumber\\
&=& \left | \begin{matrix}
1+e^{2\xi_2-\psi_2}& e^{\xi_1+\xi_2} & e^{\xi_2}\\
k_2e^{2\xi_2-\psi_2}-k_2 & k_1e^{\xi_1+\xi_2} & 0 \\
2k_2^2e^{2\xi_2-\psi_2} & k_1(k_1+k_2)e^{\xi_1+\xi_2} & 0 \\
\end{matrix} \right |\nonumber\\
&=& k_1k_2e^{\xi_1+2\xi_2}[(k_1-k_2)e^{2\xi_2-\psi_2}-(k_1+k_2)]\\
&=& k_1k_2(k_1+k_2)e^{\xi_1+2\xi_2}(e^{2\xi_2-\psi_2+h}-1).
\end{eqnarray*}
Substituting the quantities $J_1$, $J_2$ above and \eqref{4.14} and \eqref{4.17} into \eqref{4.13}, we have
\begin{eqnarray*}
 \tilde{W}r(y)&=& \alpha_2(1+e^{\xi_2})^{-3}(1+e^{\xi_2-\psi_2})^{-3}(\alpha_1J_1+ \beta J_2)+e.s.t.\nonumber\\
&=& \alpha_2(1+e^{\xi_2})^{-3}(1+e^{\xi_2-\psi_2})^{-3}[\alpha_1k_1k_2(k_1+k_2)(e^{2\xi_2-\psi_2}+1)(e^{2\xi_2-\psi_2+h}-1)\nonumber\\
& & +\beta k_1k_2(k_1+k_2)e^{\xi_1+2\xi_2}(e^{2\xi_2-\psi_2+h}-1)]+e.s.t.\\
&=& \alpha_2k_1k_2(k_1+k_2)(1+e^{\xi_2})^{-3}(1+e^{\xi_2-\psi_2})^{-3}e^{\xi_1+\xi_2}[\alpha_1e^{h}(e^{2\xi_2-\psi_2}+1)\nonumber\\
& & +\beta e^{\xi_2}](q_2-1)+e.s.t.,\nonumber
\end{eqnarray*}
where $q_2\triangleq e^{2\xi_2-\psi_2+h}$ is always positive. Note that $\tilde{c}_1>\tilde{c}_2$, then one can readily confirm through \eqref{4.9}, \eqref{4.10} and \eqref{4.16} that
$$
k_1>k_2>0, \alpha_i<0,\;i=1,2, \beta<0.
$$
Consequently, one has
$$
\alpha_1e^{h}(e^{2\xi_2-\psi_2}+1)+\beta e^{\xi_2}<0.
$$
Thus, this implies that the leading term of $ \tilde{W}r(y)$ has a unique zero determined
by $q_2 = 1$.

{\sl Case II: Assume that $\tau\gg 1$ and $(y-\tilde{c}_1\tau)=O(1)$.}
In this case, it follows from $\tilde{c}_1>\tilde{c}_2$ that $\xi_2=k_2(y-\tilde{c_2}\tau+y_{20})\rightarrow +\infty$. This
implies that $e^{-\xi_2}\ll 1$, allowing us to omit many terms in \eqref{4.6}. In particular, we find that up to exponentially small terms
{ \begin{eqnarray}\label{4.19}
\left(\ln\!\left(\frac{g}{f}\right)\right)_{y\xi_1}&=& \left(\frac{(e^{-\psi_1}-1)e^{\xi_1+2h}}{(1+e^{\xi_1+2h})(1+e^{\xi_1-\psi_1+2h})}\right)_y \nonumber\\
& \triangleq & (\hat{f}_1)_y,
\end{eqnarray}
\begin{eqnarray}\label{4.20}
\left(\ln\!\left(\frac{g}{f}\right)\right)_{y\xi_2}&=& \left(\frac{[(1+e^{\xi_1})(1+e^{\xi_1-\psi_1+2h})-(1+e^{\xi_1-\psi_1})(e^{\psi_2}+e^{\xi_1+\psi_2+2h})]e^{-\xi_2}}{(1+e^{\xi_1+2h})(1+e^{\xi_1-\psi_1+2h})}\right)_y \nonumber\\
& \triangleq & (\hat{f}_2)_y.
\end{eqnarray}}
Thus, it follows from \eqref{4.6}, \eqref{4.19} and \eqref{4.20} that
\begin{equation}\label{4.21}
\tilde{W}r(y)=Wr_2((\hat{f}_1)_y,(\hat{f}_2)_y)+e.s.t..
\end{equation}
Similarly, by the homogeneity property of Wronskians, we obtain
\begin{eqnarray}\label{4.22}
 Wr_2((\hat{f}_1)_y,(\hat{f}_2)_y)&:=&  Wr_3(1,\hat{f}_1,\hat{f}_2)\\
&=& (1+e^{\xi_1+2h})^{-3}(1+e^{\xi_1-\psi_1+2h})^{-3}Wr_3\left((1+e^{\xi_1+2h})(1+e^{\xi_1-\psi_1+2h}),\hat{g}_1,\hat{g}_2\right)\nonumber,
\end{eqnarray}
where
\eqnn{
  \hat{g}_1&:=(e^{-\psi_1+2h}-e^{2h})e^{\xi_1},\\
\hat{g}_2&:= [(1+e^{\xi_1})(1+e^{\xi_1-\psi_1+2h})-(1+e^{\xi_1-\psi_1})(e^{\psi_2}+e^{\xi_1+\psi_2+2h})]e^{-\xi_2} \nonumber\\
&= (1-e^{\psi_2})e^{-\xi_2}+(1+e^{-\psi_1+2h}-e^{\psi_2+2h}-e^{\psi_2-\psi_1})e^{\xi_1-\xi_2}+(e^{-\psi_1+2h}-e^{-\psi_1+\psi_2+2h})e^{2\xi_1-\xi_2}.\nonumber
}
 For the sake of convenience, we denote
\begin{equation}\label{4.24}
 \hat{\alpha}:=e^{-\psi_1+2h}-e^{2h},\quad \hat{\beta}:=1-e^{\psi_2},
\end{equation}
and
\begin{equation}\label{4.25}
 \hat{\gamma}:=1+e^{-\psi_1+2h}-e^{\psi_2+2h}-e^{\psi_2-\psi_1},\quad \hat{\delta}:=e^{-\psi_1+2h}-e^{-\psi_1+\psi_2+2h}.
\end{equation}

Noting that $\hat{g}_1=\hat{\alpha} e^{\xi_1}$, we find after some manipulations
\begin{eqnarray}\label{4.26}
& & Wr_3\left((1+e^{\xi_1+2h})(1+e^{\xi_1-\psi_1+2h}),\hat{g}_1,\hat{g}_2\right)\nonumber \\
&=& \hat{\alpha}[\hat{\beta}Wr_3\left(1+e^{2\xi_1-\psi_1+4h}, e^{\xi_1},e^{-\xi_2}\right)+\hat{\gamma}Wr_3\left(1+e^{2\xi_1-\psi_1+4h}, e^{\xi_1},e^{\xi_1-\xi_2}\right)\nonumber\\
& & +\hat{\delta}Wr_3\left(1+e^{2\xi_1-\psi_1+4h}, e^{\xi_1},e^{2\xi_1-\xi_2}\right)]\\
&\triangleq & \hat{\alpha}[\hat{\beta}\hat{J}_1+\hat{\gamma} \hat{J}_2+\hat{\delta} \hat{J}_3].\nonumber
\end{eqnarray}
Recall that $e^h=\frac{k_1-k_2}{k_1+k_2}$, by using properties of determinants, one has
\begin{eqnarray*}
 \hat{J}_1&=& Wr_3\left(1+e^{2\xi_1-\psi_1+4h}, e^{\xi_1},e^{-\xi_2}\right)\nonumber\\
&=& \left | \begin{matrix}
1+e^{2\xi_1-\psi_1+4h}& e^{\xi_1} & e^{-\xi_2}\\
2k_1e^{2\xi_1-\psi_1+4h} & k_1e^{\xi_1} & -k_2e^{-\xi_2} \\
4k_1^2e^{2\xi_1-\psi_1+4h} & k_1^2e^{\xi_1} & k_2^2e^{-\xi_2} \\
\end{matrix} \right |\nonumber\\
&=& \left | \begin{matrix}
1+e^{2\xi_2-\psi_2}& e^{\xi_1} & e^{-\xi_2}\\
k_1e^{2\xi_1-\psi_1+4h}-k_1 & 0 & -(k_1+k_2)e^{-\xi_2} \\
2k_1^2e^{2\xi_1-\psi_1+4h} & 0 & k_2(k_1+k_2)e^{-\xi_2}\\
\end{matrix} \right |\nonumber\\
&=& -(k_1+k_2)e^{\xi_1-\xi_2}[(k_1k_2+2k_1^2)e^{2\xi_1-\psi_1+4h}-k_1k_2].
\end{eqnarray*}
Similarly
\begin{eqnarray*}
 \hat{J}_2&=& Wr_3\left(1+e^{2\xi_1-\psi_1+4h}, e^{\xi_1},e^{\xi_1-\xi_2}\right)\nonumber\\
&=& \left | \begin{matrix}
1+e^{2\xi_1-\psi_1+4h}& e^{\xi_1} & e^{\xi_1-\xi_2}\\
2k_1e^{2\xi_1-\psi_1+4h} & k_1e^{\xi_1} & (k_1-k_2)e^{\xi_1-\xi_2} \\
4k_1^2e^{2\xi_1-\psi_1+4h} & k_1^2e^{\xi_1} & (k_1-k_2)^2e^{\xi_1-\xi_2} \\
\end{matrix} \right |\nonumber\\
&=& \left | \begin{matrix}
1+e^{2\xi_2-\psi_2}& e^{\xi_1} & e^{\xi_1-\xi_2}\\
k_1e^{2\xi_1-\psi_1+4h}-k_1 & 0 & -k_2e^{\xi_1-\xi_2} \\
2k_1^2e^{2\xi_1-\psi_1+4h} & 0 & -k_2(k_1-k_2)e^{\xi_1-\xi_2}\\
\end{matrix} \right |\nonumber\\
&=& -k_1k_2e^{2\xi_1-\xi_2}[(k_1+k_2)e^{2\xi_1-\psi_1+4h}+(k_1-k_2)].
\end{eqnarray*}
\begin{eqnarray*}
 \hat{J}_3&=& Wr_3\left(1+e^{2\xi_1-\psi_1+4h}, e^{\xi_1},e^{2\xi_1-\xi_2}\right)\nonumber\\
&=& \left | \begin{matrix}
1+e^{2\xi_1-\psi_1+4h}& e^{\xi_1} & e^{2\xi_1-\xi_2}\\
2k_1e^{2\xi_1-\psi_1+4h} & k_1e^{\xi_1} & (2k_1-k_2)e^{2\xi_1-\xi_2} \\
4k_1^2e^{2\xi_1-\psi_1+4h} & k_1^2e^{\xi_1} & (2k_1-k_2)^2e^{2\xi_1-\xi_2} \\
\end{matrix} \right |\nonumber\\
&=& \left | \begin{matrix}
1+e^{2\xi_2-\psi_2}& e^{\xi_1} & e^{2\xi_1-\xi_2}\\
k_1e^{2\xi_1-\psi_1+4h}-k_1 & 0 & (k_1-k_2)e^{2\xi_1-\xi_2} \\
2k_1^2e^{2\xi_1-\psi_1+4h} & 0 & (k_1-k_2)(2k_1-k_2)e^{2\xi_1-\xi_2}\\
\end{matrix} \right |\nonumber\\
&=& -k_1(k_1-k_2)e^{3\xi_1-\xi_2}[(k_2-2k_1)-k_2e^{2\xi_1-\psi_1+4h}].
\end{eqnarray*}
It is observed from \eqref{4.24} and \eqref{4.25} that
$$\hat{\delta}=(1-e^{\psi_2})e^{-\psi_1+2h}=\hat{\beta}e^{-\psi_1+2h}.
$$
Consequently, noting that $e^{2h}=(\frac{k_1-k_2}{k_1+k_2})^2$, one has
\begin{eqnarray}\label{4.27}
& & \hat{\beta}\hat{J}_1+\hat{\gamma}\hat{J}_2+\hat{\delta}\hat{J}_3\nonumber \\
&=& -\hat{\beta}(k_1+k_2)e^{\xi_1-\xi_2}[(k_1k_2+2k_1^2)e^{2\xi_1-\psi_1+4h}-k_1k_2]-\hat{\beta}(k_1-k_2)e^{\xi_1-\xi_2}e^{2\xi_1-\psi_1+2h}\nonumber\\
& & \cdot[(k_1k_2-2k_1^2)-k_1k_2e^{2\xi_1-\psi_1+4h}]-\hat{\gamma}k_1k_2e^{2\xi_1-\xi_2}[(k_1+k_2)e^{2\xi_1-\psi_1+4h}+(k_1-k_2)]\nonumber\\
&=& \hat{\beta}e^{\xi_1-\xi_2}\cdot \frac{k_1k_2(k_1+k_2)}{k_1-k_2} (e^{2\xi-\psi_1+4h}+1)[(k_1+k_2)e^{2\xi-\psi_1+4h}+(k_1-k_2)]\nonumber\\
& & -\hat{\gamma}k_1k_2e^{2\xi_1-\xi_2}[(k_1+k_2)e^{2\xi_1-\psi_1+4h}+(k_1-k_2)]\nonumber\\
&=& k_1k_2e^{\xi_1-\xi_2}[(k_1+k_2)e^{2\xi_1-\psi_1+4h}+(k_1-k_2)]\cdot [\hat{\beta}(e^{2\xi_1-\psi_1+3h}+e^{-h})-\hat{\gamma}e^{\xi_1}].
\end{eqnarray}
Substituting \eqref{4.22}, \eqref{4.26} and \eqref{4.27} into \eqref{4.21}, we have
\begin{eqnarray}\label{4.18}
& & \tilde{W}r(y)\nonumber \\
&=& (1+e^{\xi_1+2h})^{-3}(1+e^{\xi_1-\psi_1+2h})^{-3}\hat{\alpha}k_1k_2e^{\xi_1-\xi_2}[(k_1+k_2)e^{2\xi_1-\psi_1+4h}+(k_1-k_2)]\nonumber\\
& & \cdot [\hat{\beta}(e^{2\xi_1-\psi_1+3h}+e^{-h})-\hat{\gamma}e^{\xi_1}]+e.s.t..\nonumber
\end{eqnarray}
Note that $\tilde{c}_1>\tilde{c}_2$, then one can readily confirm through \eqref{4.9}, \eqref{4.10}, \eqref{4.24}, and \eqref{4.25} that
$$
k_1>k_2>0, \quad \hat{\alpha}=e^{2h}(e^{-\psi_1}-1)<0,\quad \hat{\beta}=1-e^{\psi_2}<0,
$$
and
$$
\hat{\gamma}=1+e^{-\psi_1+2h}-e^{\psi_2+2h}-e^{\psi_2-\psi_1}=\frac{4\kappa k_2(k_1-k_2)}{(1+\kappa k_1)(1-\kappa k_2)(k_1+k_2)}>0.
$$
Thus, this implies that $ \tilde{W}r(y)$ is always positive.

{\sl Case III: Assume that $(y-\tilde{c}_2\tau)$ is large and positive and $(y-\tilde{c}_1\tau)$ is large and negative.}
In this case, it implies that $e^{\xi_1}\ll 1$ and $e^{-\xi_2}\ll 1$. Thus, we can find that
\begin{equation*}
\left(\ln\!\left(\frac{g}{f}\right)\right)_{y\xi_1}= k_1(e^{-\psi_1}-1)e^{\xi_1+2h}+e.s.t.,
\end{equation*}
 \begin{equation*}
\left(\ln\!\left(\frac{g}{f}\right)\right)_{y\xi_2}= k_2(e^{\psi_2}-1)e^{-\xi_2}+e.s.t..
\end{equation*}
Consequently,
\begin{equation*}
\left(\ln\!\left(\frac{g}{f}\right)\right)_{y\xi_1y}= k_1^2(e^{-\psi_1}-1)e^{\xi_1+2h}+e.s.t.,
\end{equation*}
 \begin{equation*}
\left(\ln\!\left(\frac{g}{f}\right)\right)_{y\xi_2y}= -k_2^2(e^{\psi_2}-1)e^{-\xi_2}+e.s.t..
\end{equation*}
Note that $e^{\psi_2}-1>0$ and $1-e^{-\psi_1}>0$, it then follows that
{ \begin{eqnarray*}
\tilde{W}r(y)&=& \left(\ln\!\left(\frac{g}{f}\right)\right)_{y\xi_1}\left(\ln\!\left(\frac{g}{f}\right)\right)_{y\xi_2y}-\left(\ln\!\left(\frac{g}{f}\right)\right)_{y\xi_2}\left(\ln\!\left(\frac{g}{f}\right)\right)_{y\xi_1y} \nonumber\\
&=& k_1k_2(k_1+k_2)(e^{\psi_2}-1)(1-e^{-\psi_1})e^{\xi_1-\xi_2+2h}+e.s.t.,
\end{eqnarray*}}
which is always positive.

In summary, the Wronskian $\tilde{W}r(y)$ changes sign only once, and that for large positive
times, the zero occurs approximately near $\xi_2\approx \frac{\psi_2-h}{2}$. This completes the proof of Proposition \ref{pro4.1}.
\end{proof}

\section{2-soliton as a constrained minimizer}
\label{5s}
In this section, we aim to show that
$\hat{\mu}$ is a constrained (non-isolated) minimizer of $F$, which implies stability by the result of \cite{MS93}.

 Note that the spectral properties of $\mathcal{L}$ in Section 4 have already shown that

(a) The essential spectrum of $\mathcal{L}$ is bounded below by a positive constant.

(b) $\mathcal{L}$ has only one negative eigenvalue and its kernel is two-dimensional,
spanned by the tangent space of the family of critical points. \\
This also implies that the 2-soliton profiles $\tilde{\mu}$ are nondegenerate critical points of the Lagrangian $F$.
However, it is well known that $\tilde{\mu}$ may be a local constrained minimum while $\mathcal{L}(\tilde{\mu})$ has negative eigenvalues.
Fortunately, for nondegenerate critical points, it was shown by Hestenes \cite{He66, He75} that the condition
\textcolor{black}
{\begin{equation}\label{5.1}
 \langle \eta,\mathcal{L}\eta\rangle>0,\quad \forall \eta\in\{\eta\in L^2: \langle \eta,\frac{\delta F_i}{\delta m}(\tilde{\mu})\rangle=0,i=1,2, \eta\in R(\mathcal{L})\}
\end{equation}}
 is sufficient to guarantee that $\tilde{\mu}$ is a constrained (non-isolated) minimizer, where $R(\mathcal{L})$ is the range of the operator $\mathcal{L}$ in $L^2$.
 However, the direct verification of the condition is often extremely intricate. In order to address this difficulty, we need to use an indirect approach developed by Maddocks and Sachs (see
\cite[Lemma 2.1]{MS93}).

\begin{proposition} \label{pro5.1}
Consider the functional
 \begin{equation*}
F(m)=F_3(m)+\lambda_1F_1(m)+\lambda_2F_2(m).
\end{equation*}
If a nondegenerate extremal $m$ can be embedded in a family of extremals of $F$, then the constrained second-order condition {\eqref{5.1}} is satisfied
if and only if the number of positive eigenvalues of the Hessian matrix
\eq{
M:=\left\{\frac{\partial^2 F(\tilde{\mu})}{\partial\lambda_i\partial\lambda_j}\right\}_{2\times 2}
}{M}
equals to the number of negative eigenvalues of $\mathcal{L}:=\frac{\delta^2 F}{\delta m^2}(\tilde{\mu})$.
\end{proposition}

 Based on Proposition \ref{pro5.1}, we show that the 2-soliton solution is a constraint minimizer (i.e. Condition \ref{5.1} is satisfied) by proving the following lemma.

\begin{lemma} \label{le5.1}
At the values
\eq{\lambda_1&=-\frac{1}{16\kappa^4}+\frac{1}{2\kappa^2}\left(\frac{1}{c_1-\kappa^2}+\frac{1}{c_2-\kappa^2}\right)+\frac{2}{(c_1-\kappa^2)(c_2-\kappa^2)},\\
\lambda_2&=-\left(\frac{1}{4\kappa^2}+\frac{1}{c_1-\kappa^2}+\frac{1}{c_2-\kappa^2}\right),}{l1l2}
the Hessian matrix defined in \eqref{M}
has precisely one positive eigenvalue and one negative eigenvalue.
\end{lemma}

\begin{proof}
 Firstly, we compute the invariants $F_1$ and $F_2$ at the single soliton $\mu_c$.
The explicit forms of $F_1(\mu_c)$ can be found in \cite{DLL25}
 \eqnn{
F_1(\mu_c)=\frac{8}{\kappa}\left[\ln\frac{\sqrt{c-\kappa^2}+\sqrt{c-3\kappa^2}}{\sqrt{2}\kappa}-\sqrt{\frac{c-3\kappa^2}{c-\kappa^2}}\right].
}
We now provide the computation of $F_2(\mu_c)$. In view of \eqref{3.1}, \eqref{3.13} and \eqref{3.14}, one has
\begin{eqnarray}\label{5.3}
 F_2(\mu_c) &=& \int_{\mathbb{R}}\left(\frac{\mu_{c,x}^2}{\mu_c^5}+\frac{1}{4\mu_c^3}+\frac{3}{4\kappa^4}\mu_c-\frac{1}{\kappa^3}\right)\,dx \\
&=& \int_{\mathbb{R}}\left[\left(4A^{-2}\phi_{c,x}^2+\frac{3}{4\kappa^4}\right)\cdot \frac{A}{\sqrt{E-4A\phi_c}}+\frac{(E-4A\phi_c)^{\frac{3}{2}}}{4A^3}-\frac{1}{\kappa^3}\right]\,dx. \nonumber
\end{eqnarray}

We rescale the smooth solitary wave solutions $\phi_c$ as
\begin{equation}\label{5.4}
 \phi_c=\kappa+\frac{c-\kappa^2}{4\kappa} \varphi.
\end{equation}
 By Lemma \ref{le3.1} in \cite{DLL25}, we have
\eq{\varphi\in (0,\varphi_0]\subset(0,1),\;\;\varphi_0:=\sup_{\xi\in\mathbb{R}}{(\varphi)}=\frac{4\kappa\left(\sqrt{2(c-\kappa^2)}-2\kappa\right)}{c-\kappa^2},}{p0def}
 and
it is easily seen that $\varphi_0\rightarrow 1$ as $\kappa\rightarrow \frac{\sqrt{c}}{3}$. Recall that $A=\kappa(c-\kappa^2), E=(c-\kappa^2)(c+3\kappa^2)$, and substituting \eqref{5.4} into
\eqref{3.10}, we have\textcolor{black}{
\eqnn{
\varphi_{x}^2=\varphi^2-r(1-\sqrt{1-\varphi})^2,
}
where $r=\frac{8k^2}{c-k^2}$ and $\varphi\in(0,1)$. Note that for $x\in(-\infty,0)$, $\varphi_x>0$, so we have
 \begin{equation}\label{5.6}
 \varphi_{x}=\sqrt{\varphi^2-r\left(1-\sqrt{1-\varphi}\right)^2},\quad x\in(-\infty,0).
\end{equation}
Consequently, substituting \eqref{5.4} and \eqref{5.6} into \eqref{5.3} leads to
\begin{eqnarray*}
 F_2(\mu_c) &=& \frac{1}{2\kappa^3}\int_{-\infty}^{0}\left[\frac{3+\varphi^2-r(1-\sqrt{1-\varphi})^2+(1-\varphi)^2-4\sqrt{1-\varphi}}{\sqrt{1-\varphi}}\right]\,dx \\
&=& \frac{1}{2\kappa^3}\int_{0}^{\varphi_0}\left[\frac{3+\varphi^2-r(1-\sqrt{1-\varphi})^2+(1-\varphi)^2-4\sqrt{1-\varphi}}{\sqrt{1-\varphi}}\right]\\
& &\cdot \frac{d\varphi}{\sqrt{\varphi^2-r(1-\sqrt{1-\varphi})^2}}. \nonumber
\end{eqnarray*}
Making the change of variable
$$ z\triangleq \sqrt{1-\varphi}, \quad \varphi=1-z^2,\quad d\varphi=-2zdz,$$
with $r=\frac{8\kappa^2}{c-\kappa^2}$ and  $\varphi_0$ given in \eqref{p0def}, the expression of $ F_2(\mu_c)$ becomes
 \begin{eqnarray*}
 F_2(\mu_c) &=& \frac{1}{2\kappa^3}\int_{1}^{\sqrt{1-\varphi_0}}\left[\frac{3+(1-z^2)^2-r(1-z)^2+z^4-4z}{z}\right]\cdot \frac{-2z}{\sqrt{(1-z^2)^2-r(1-z)^2}}dz \nonumber\\
&=&   \frac{1}{\kappa^3}\int_{\sqrt{1-\varphi_0}}^{1}\left[\frac{4-r(1-z)-2z^2(1+z)}{\sqrt{(1+z)^2-r}}\right]dz\nonumber\\
&=&   \frac{1}{\kappa^3}\Big[4\ln(z+1+\sqrt{(z+1)^2-r})-(r-2z^2-2(z+1))\sqrt{(z+1)^2-r}\nonumber\\
& & +\frac{4}{3}((z+1)^2-r)^{\frac{3}{2}}\Big]\Bigg|_{z=\sqrt{1-\varphi_0}}^{z=1}\\
&=&   \frac{1}{\kappa^3}\left[4\ln(2+\sqrt{4-r})+(r-6)\sqrt{4-r}+\frac{4}{3}(4-r)^{\frac{3}{2}}-4\ln(1+\sqrt{1-\varphi_0})\right]\nonumber\\
&=&  \frac{4}{\kappa^3}\left[\ln\!\left(\frac{\sqrt{c-\kappa^2}+\sqrt{c-3\kappa^2}}{\sqrt{2}\kappa}\right)-\frac{c+3\kappa^2}{3(c-\kappa^2)}\sqrt{\frac{c-3\kappa^2}{c-\kappa^2}}\right].
\end{eqnarray*}} 

Since both $F_1$ and $F_2$ are time-invariant, we can exploit the
asymptotic form of the 2-soliton in the limit $t\rightarrow \infty$ as being exponentially close
to a sum of two 1-solitons infinitely far apart. Consequently, the invariant functionals $F_i(\tilde{\mu})$ can be computed as the sum of the invariants for the two single solitons $\mu_{c_1}$ and $\mu_{c_2}$.
That is,
\eqnn{
F_i(\tilde{\mu})=F_i(\mu_{c_1})+F_i(\mu_{c_2}), \quad i=1,2.
}
Thus
\begin{eqnarray}\label{5.8}
 F_1(\tilde{\mu}) &=& \frac{8}{\kappa}\Bigg[\ln\left(\frac{\sqrt{c_1-\kappa^2}+\sqrt{c_1-3\kappa^2}}{\sqrt{2}\kappa}\right)-\sqrt{\frac{c_1-3\kappa^2}{c_1-\kappa^2}}\nonumber\\
& & +\ln\left(\frac{\sqrt{c_2-\kappa^2}+\sqrt{c_2-3\kappa^2}}{\sqrt{2}\kappa}\right)-\sqrt{\frac{c_2-3\kappa^2}{c_2-\kappa^2}}\Bigg],
\end{eqnarray}
and
\begin{eqnarray}\label{5.9}
 F_2(\tilde{\mu}) &=& \frac{4}{\kappa^3}\Bigg[\ln\!\left(\frac{\sqrt{c_1-\kappa^2}+\sqrt{c_2-3\kappa^2}}{\sqrt{2}\kappa}\right)-\frac{c_1+3\kappa^2}{3(c_1-\kappa^2)}\sqrt{\frac{c_1-3\kappa^2}{c_1-\kappa^2}} \nonumber\\
& & + \ln\!\left(\frac{\sqrt{c_2-\kappa^2}+\sqrt{c_2-3\kappa^2}}{\sqrt{2}\kappa}\right)-\frac{c_2+3\kappa^2}{3(c_2-\kappa^2)}\sqrt{\frac{c_2-3\kappa^2}{c_2-\kappa^2}}\Bigg].
\end{eqnarray}

Finally we determine the number of positive eigenvalues of the Hessian matrix $M$. Note that
\begin{equation*}
M=\frac{\partial^2 F}{\partial\lambda_i\partial\lambda_j}=\frac{\partial}{\partial c_l}\left(\frac{\partial F}{\partial \lambda_i}\right)\cdot \frac{\partial c_l}{\partial \lambda_j},
\end{equation*}
which results in
\begin{equation}\label{Mcalc}
M=\left ( \begin{matrix}
\frac{\partial F_1}{\partial c_1} & \frac{\partial F_1}{\partial c_2} \\
\frac{\partial F_2}{\partial c_1} & \frac{\partial F_2}{\partial c_2} \\
\end{matrix} \right ) \left ( \begin{matrix}
\frac{\partial \lambda_1}{\partial c_1} & \frac{\partial \lambda_1}{\partial c_2} \\
\frac{\partial \lambda_2}{\partial c_1} & \frac{\partial \lambda_2}{\partial c_2}\\
\end{matrix} \right )^{-1}.
\end{equation}
Using \eqref{5.8}, \eqref{5.9} and the quantities of $\lambda_1, \lambda_2$, it follows that the determinant of $M$ is
\begin{equation*}
\det M=-\frac{16}{\kappa^2}\sqrt{(c_1-\kappa^2)(c_2-\kappa^2)(c_1-3\kappa^2)(c_2-3\kappa^2)}<0.
\end{equation*}
This implies that $M$ has precisely one negative and one positive eigenvalue. The proof of Lemma \ref{le5.1} is completed.
\end{proof}

\section{Proof of Theorem \ref{th1.1}}
\label{6s}

In this short section, we recapitulate the results obtained previously and explain how they are sufficient for the proof of Theorem \ref{th1.1} according to \cite{MS93}.

In Section \ref{3s}, we showed that the 2-soliton solution is a critical point of a conserved functional. In Section \ref{4s}, we showed that the 2-soliton solution {is a non-degenerate critical point}, as defined in \cite[Section 2.1]{MS93} by showing the null space of the second variation  $\mathcal{L}$ given in \eqref{4.1} is two dimensional. In Section \ref{5s}, we use Proposition \ref{pro5.1} (corresponding to \cite[Lemma 2.1]{MS93}) to show that the 2-soliton solution is a constrained (non-isolated) minimizer, i.e. Condition \eqref{5.1} is satisfied. Finally, the eigenvalues of the matrix $M$ defined in \eqref{M} do not depend on the phases $y_{i0}$, $i=1,2$, of the 2-soliton as defined in Section \ref{2.1s}, as one can see from the expression for $M$ in \eqref{Mcalc}, and the expressions for $F_i$, $i=1,2$ in \eqref{5.8} and \eqref{5.9} and for $\lambda_i$, $i=1,2$ in \eqref{l1l2}. This implies that the constant $C$ of \cite[Lemma 2.3]{MS93} is ``independent of the particular representative of the family of critical points'' \cite[last paragraph of Section 2.3]{MS93}. Theorem \ref{th1.1} then follows directly from the results of Section 2.3 of \cite{MS93}.

{\color{black}{
\section{Discussion and Conclusions}

In this paper, we have established the stability of the 2-soliton solution of the mCH equation \eqref{1} as stated in Theorem \ref{th1.1}. To do so, we obtained in Lemma \ref{le3.1} a linear combination of the known conserved quantities \eqref{7} for which the two-soliton solution is a critical point. The spectral analysis of the Hessian is performed and the issue of the number of negative eigenvalues is solved by counting the number of zeroes of the Wronskian of the elements of the kernel (see Proposition \ref{pro4.1}).

The obvious next step is to obtain the stability of the $n$-soliton solutions. In the case of the KdV done in \cite{MS93}, the analysis relies on the fact that the Lagrange multipliers are the elementary symmetric functions of the velocities and expressions for the Wronskian can be obtained with the known formulas for the multi-soliton solutions. As for peakon equations, the stability of the multi-soliton for the Camassa-Holm equation is established with nonlocal conservation laws in \cite{Wang2022} and local ones in \cite{Zhang2025}. In both cases, the Lagrange multipliers are obtained to be symmetric function of the velocities (\cite[Proposition 2.1]{Wang2022} and \cite[Lemma 3.5]{Zhang2025}). In \cite{Zhang2025}, a distinction is made between the ``dependent'' conservation laws obtained through  a recursion relation, and the ``independent'' ones obtained by formulas written in terms of the scattering data (see \cite[Definition 2.1]{Zhang2025} taken from \cite[Equations (192) and (104)]{Constantin2007}). The independent conserved quantities, which are Fr\'echet differentiable,  are shown to be linear combinations of the dependent ones.

In our case, it is our understanding that the quantities given in \eqref{7} would correspond to the dependent conserved quantities, while the ones given in \eqref{3.1} are the independent ones. What makes our results challenging to generalize to the multi-soliton case is that we do not have a general formula that would give us the Lagrange multipliers as symmetric functions of the velocities. So, to extend our work  both the conserved quantities and the Lagrange mutipliers have to be obtained every time the number of solitons goes up by one. However, it has been shown in \cite{Zhang2025} for the Camassa-Holm equation that
the multipliers for the $n$-solitons given by Lemma 3.5 of \cite{Zhang2025} are indeed symmetric in the
quantities $1/c_i$, $i=1,...,n$. This inspires us further explore ways to appropriately combine the dependent conservation laws for the $n$-soliton case of the mCH equation \eqref{1}, and construct new independent
conservation laws, so that the corresponding Lagrange multipliers
become symmetric functions with respect to the wave velocity, which is the key point in the study of the stability of the $n$-soliton solutions, as mentioned above. We are thus working in extending our work to the general case of $n$-solitons.

}}

\vspace{5mm}

\noindent{\bf Acknowledgements}

{\color{black}{This research was partially supported by the NSFC (No.12571172)}} and the
Scientific Research Fund of Hunan Provincial Education Department
(No.21A0414). The research of S. Lafortune was supported by a Collaboration Grants for Mathematicians from the Simons Foundation (award \# 420847).
 {\color{black}{The research of Z. Liu was supported by the NSFC (No.12571188)}} and the Fundamental Research Funds for the Central
Universities, China University of Geosciences (Nos.CUGST2), and Guangdong Basic and Applied Basic Research
Foundation (Nos.2023A1515011679; 2024A1515012704).

\appendix
\renewcommand{\appendixname}{Appendix}
 \section{Conservation of the functional $E_4$} 
 \label{AA}
{{In this Appendix, we show explicitly that the quantity $E_4$ given in \eqref{7} is time-independent for the mCH equation \eqref{1}}}

It follows from \eqref{1} that
\eq{
 m_t&=-[(u^2-u_x^2)m_x+2m^2u_x],\\
 m_{tx}&=-[(u^2-u_x^2)m_{xx}+6u_xmm_x+2m^2(u-m)],\\
 m_{txx}&=-[(u^2-u_x^2)m_{xxx}+8u_xmm_{xx}+6m_x^2u_x+10m m_x(u-m)+2m^2(u_x-m_x)].
}{2.17}

A straightforward calculation shows that
\begin{eqnarray*}
\frac{d}{dt}E_4(m(t))&=& \frac{d}{dt}\int_{\mathbb{R}}\frac{m_{xx}^2}{2m^7}dx+\frac{d}{dt}\int_{\mathbb{R}}\frac{5m_x^2}{4m^7}dx+\frac{d}{dt}\int_{\mathbb{R}}-\frac{7m_x^4}{2m^9}dx+\frac{d}{dt}\int_{\mathbb{R}}\frac{1}{16m^5}dx \\
&=& I_1+I_2+I_3+I_4.
\end{eqnarray*}
In view of \eqref{2.17}, and integrating by parts several times leads to
\begin{eqnarray*}
I_1&=&  \frac{d}{dt}\int_{\mathbb{R}}\frac{m_{xx}^2}{2m^7}dx\\
&=& \int_{\mathbb{R}} \left(\frac{m_{xx}m_{xxt}}{m^7}-\frac{7m_{xx}^2m_t}{2m^8}\right)dx\\
&=& \int_{\mathbb{R}} \left(\frac{-14m_x^4u_x}{m^8}-\frac{28m_x^3(u-m)}{m^7}-\frac{5m_x^2u_x}{m^6}+\frac{2m_x(u-m)}{m^5}\right)dx.
\end{eqnarray*}
Similarly,
\begin{eqnarray*}
I_2&=&  \frac{d}{dt}\int_{\mathbb{R}}\frac{5m_x^2}{4m^7}dx\\
&=& \int_{\mathbb{R}}  \left(\frac{5m_x m_{tx}}{2m^7}-\frac{35m_x^2 m_t}{4m^8}\right)dx\\
&=&  \int_{\mathbb{R}} \left(\frac{5m_x^2u_x}{m^6}-\frac{5m_x(u-m)}{m^5}\right)dx,
\end{eqnarray*}
\begin{eqnarray*}
I_3&=&  \frac{d}{dt}\int_{\mathbb{R}}-\frac{7m_x^4}{2m^9}dx\\
&=& \int_{\mathbb{R}}  \left(\frac{-14m_x^3 m_{tx}}{m^9}+\frac{63m_x^4 m_t}{2m^{10}}\right)dx\\
&=&  \int_{\mathbb{R}} \left(\frac{14m_x^4 u_x}{m^8}+\frac{28m_x^3(u-m)}{m^7}\right)dx,
\end{eqnarray*}

\begin{eqnarray*}
I_4&=&  \frac{d}{dt}\int_{\mathbb{R}}\frac{1}{16m^5}dx\\
&=&  \frac{d}{dt}\int_{\mathbb{R}}\frac{-5m_t}{16m^6}dx \\
&=&  \int_{\mathbb{R}} \left(\frac{5[(u^2-u_x^2)m_x+2m^2u_x]}{16m^6}\right)dx.
\end{eqnarray*}

Thus, we can obtain
\begin{eqnarray*}
I_1+I_2+ I_3+I_4 &=& \int_{\mathbb{R}} \left(\frac{-3m_x(u-m)}{m^5}+\frac{5(u^2-u_x^2)m_x}{16m^6}+\frac{5u_x}{8m^4}\right)dx\\
&=& \int_{\mathbb{R}}  \left(\frac{3u}{4m^4}-\frac{1}{m^3}-\frac{u^2-u_x^2}{16m^5}\right)_xdx.\\
&=& 0.
\end{eqnarray*}

Consequently, it yields that $\displaystyle{\frac{d}{dt}E_4(m(t))=0}$.

\end{document}